\documentclass[a4paper,11pt]{amsart}
\usepackage[left=2.7cm,right=2.7cm,top=3.5cm,bottom=3cm]{geometry}
\usepackage{amsthm, amssymb,mathrsfs,mathtools,color, enumerate,fullpage, tikz-cd}
\usepackage[pagebackref]{hyperref}
\usepackage[capitalize, noabbrev]{cleveref}

\newtheorem{definition}{Definition}[section]

\newtheorem{proposition}[definition]{Proposition}
\newtheorem{corollary}[definition]{Corollary}
\newtheorem{assumption}[definition]{Assumption}
\newtheorem{theorem}[definition]{Theorem}
\newtheorem{conjecture}[definition]{Conjecture}
\newtheorem{question}[definition]{Question}

\theoremstyle{remark}
\newtheorem{remark}[definition]{Remark}

\title{Eisenstein degeneration of Euler systems}

\newcommand{\into}{\hookrightarrow}

\newcommand{\ZZ}{\mathbf{Z}}
\newcommand{\QQ}{\mathbf{Q}}
\newcommand{\CC}{\mathbf{C}}
\newcommand{\Qp}{\QQ_p}
\newcommand{\Qpb}{\overline{\QQ}_p}
\newcommand{\Qb}{\overline{\QQ}}
\newcommand{\Zp}{\ZZ_p}

\newcommand{\cF}{\mathcal{F}}
\newcommand{\cO}{\mathcal{O}}
\newcommand{\cR}{\mathcal{R}}
\newcommand{\cH}{\mathcal{H}}
\newcommand{\cW}{\mathcal{W}}

\newcommand{\hE}{{\mathbf{E}}}
\newcommand{\hf}{{\mathbf{f}}}
\newcommand{\hg}{{\mathbf{g}}}

\newcommand{\hj}{{\mathbf j}}
\newcommand{\hk}{{\mathbf k}}

\newcommand{\Dcris}{\mathbf{D}_{\mathrm{cris}}}

\renewcommand{\ge}{\geqslant}
\renewcommand{\le}{\leqslant}

\DeclareMathOperator{\GL}{GL}
\DeclareMathOperator{\PR}{PR}

\DeclareMathOperator{\Fil}{Fil}
\DeclareMathOperator{\Frob}{Frob}
\DeclareMathOperator{\Gal}{Gal}
\DeclareMathOperator{\crit}{crit}
\DeclareMathOperator{\Hom}{Hom}

\DeclareMathOperator{\loc}{loc}

\newcommand{\eps}{\varepsilon}
\newcommand{\rig}{\mathrm{rig}}

\newcommand{\et}{\text{\textup{\'et}}} % upright font even in theorem environments

\usepackage{tabto}

\begin{document}

\title[Eisenstein degeneration of Beilinson--Kato classes]{Eisenstein degeneration of Beilinson--Kato classes and circular units}

\author{Javier Polo and \'Oscar Rivero  \\
(\MakeLowercase{with an appendix by} Ju-Feng Wu)}

\begin{abstract}
The aim of this note is to explore the Euler system of Beilinson--Kato elements in families passing through the critical $p$-stabilization of an Eisenstein series attached to two Dirichlet characters $(\psi,\tau)$, a setting which is not covered by the recent works of Benois and B\"uy\"ukboduk. In this context, we establish an explicit connection with the system of circular units, utilizing suitable factorization formulas in a situation where several of the $p$-adic $L$-functions vanish. In that regard, our main results may be seen as an Euler system incarnation of the factorization formula of Bella\"iche and Dasgupta. One of the most significant aspects is that, depending on the parity of $\psi$ and $\tau$, different phenomena arise; while some can be addressed with our methods, others pose new questions. Finally, we discuss analogous results in the framework of Beilinson--Flach classes.
\end{abstract}

\date{\today}

\address{J. P.: Departamento de Matem\'aticas (Universidade de Santiago de Compostela), Santiago de Compostela, Spain}
\email{javier.polo.noche@rai.usc.es}

\address{O. R.: CITMAga and Departamento de Matem\'aticas (Universidade de Santiago de Compostela), Santiago de Compostela, Spain}
\email{oscar.rivero@usc.es}

\address{J.-F. W.: School of Mathematics and Statistics (University College Dublin), Dublin 4, Ireland}
\email{ju-feng.wu@ucd.ie}

\maketitle

\setcounter{tocdepth}{1}
\tableofcontents

\section{Introduction}

The study of Eisenstein series has played a prominent role in the theory of modular forms. One setting that has traditionally been elusive within the framework of Coleman families has to do with the so-called {\it critical Eisenstein series}, which refers to the case where one considers the family corresponding to the critical slope $p$-stabilization of an Eisenstein series. The main result summarizing the construction of families of critical $p$-adic $L$-functions is presented in \cite{bellaiche11a}, where the author discusses the conditions under which such $p$-adic $L$-functions exist, using the eigencurve of modular symbols. The Eisenstein case was later explored in \cite{bellaichedasgupta15}, where the authors established a factorization formula valid for half of the weight space that satisfies a sign condition. In fact, this is one of the key inputs for this work. The $p$-adic $L$-function vanishes in the other half of the weight space, raising the natural question about the derivative along the weight direction in an attempt to capture its arithmetic information.

The theory of critical modular forms and their connection with Euler systems has been explored in the works of Benois and B\"uy\"ukboduk \cite{BB22}, \cite{BB24}, but under the assumption that the critical point corresponds to a cuspidal modular form. In this setting, ramification is measured in terms of the vanishing of the adjoint $p$-adic $L$-function. However, in the Eisenstein case, there is no expected ramification, although the adjoint $p$-adic $L$-function still vanishes.

The Eisenstein case was considered by Loeffler and the second author in \cite{LR1}, where they discussed what they called the {\it Eisenstein degeneration of Euler systems}. This method, which involved using the factorization of $p$-adic $L$-functions to establish connections among different types of Euler systems, was successfully applied in various contexts: diagonal cycles vs Beilinson--Flach classes, Beilinson--Flach vs Beilinson--Kato, and Heegner cycles vs elliptic units. A key step in their argument was the use of the local properties of the {\it big} Euler systems, which enabled a comparison between $p$-adic $L$-functions.

However, these local properties cannot be used when trying to compare the Beilinson--Kato system with circular units, as in this case no local conditions are involved at $p$. Therefore, this work should be seen as a natural continuation and an extension of the methods developed in \cite{LR1}, where we modify our previous approach to cover the Beilinson--Kato case. This case turns out to be ostensibly different, showing in an easier instance some of the properties already envisaged in \cite{BB24} about the behavior of this Euler system. Further, the main result can be interpreted as an Euler-system incarnation of the results of Bella\"iche--Dasgupta.

\subsection{Set-up}

Let $E_{r+2}(\psi,\tau) \in M_k(N, \xi)$ be an Eisenstein series of even weight $k=r+2$, level $\Gamma_1(N)$ and nebentype $\xi = \psi \tau$, where $\psi$ and $\tau$ are primitive Dirichlet characters with conductors $N_1$ and $N_2$, respectively, with $N = N_1N_2$ and $\xi(-1) = (-1)^r$. Let $\hf$ stand for the unique Coleman family passing through the critical slope $p$-stabilization of $f := E_{r+2}(\psi,\tau)$, $f_{\beta} := E_{r+2}^{\crit}(\psi,\tau)$. Unless explicitly stated otherwise, we will assume in the first part of the introduction that $r$ is even and that both $\tau$ and $\psi$ are even.

Along this article, we introduce the following assumption.
\begin{assumption}\label{ass-1}
The critical slope $p$-stabilization $f_{\beta}$ is {\it non-crticial}, i.e., the generalised eigenspace of overconvergent modular forms associated to $f_{\beta}$ is 1-dimensional. Further, the Coleman--Mazur eigencurve is smooth and locally \'etale at the point corresponding to $E_{r+2}^{\crit}(\psi,\tau)$ in the family $\hf$.
\end{assumption}

This is known to hold in a high level of generality (see e.g. \cite[\S A6]{LR1}), so this must be seen as a very mild assumption. We later discuss several equivalent formulations for this condition. From now on, let $X$ stand for a local uniformizer at that point. Similarly, we write $V(\hf)^*$ (resp. $V^c(\hf)^*$) for the non-compactly supported (resp. compactly supported) cohomology of $\hf$. Let $\mathcal H_{\Gamma}$ be the distribution algebra of the cyclotomic extension $\Gamma \cong \ZZ_p^{\times}$, i.e., the global sections of the rigid analytic space associated with $\mathrm{Spf}\, \ZZ_p[\![\Gamma]\!]$. %corresponding to the analytic functions on the $p$-adic weight space $\mathcal W = \mathrm{Hom}_{\mathrm{cont}}(\mathbf{Z}_p^{\times}, \mathbf{C}_p^{\times})$ over $\mathbf{Q}_p$. 
Let $\hj \colon \Gamma \hookrightarrow \mathcal H_{\Gamma}^{\times}$ the universal character, which we regard as a character of the Galois group by composition with the cyclotomic character.

To formulate the main result of this paper, we need to introduce the following two Euler systems.
\begin{enumerate}
\item[(a)] The {\bf circular unit} $c(\tau)(r) \in H^1(\QQ, \QQ_p(\tau^{-1})(1+r) \otimes \mathcal H_{\Gamma}(-\hj))$, constructed as an inverse limit of units along the cyclotomic tower when $\tau$ is even. For $r=0$, we just write $c(\tau)$. Further, if $r=0$ and we specialize at the trivial character, the corresponding class may be written as an explicit product of cyclotomic units.
\item[(b)] The {\bf Kato Euler system} (sometimes called Beilinson--Kato Euler system) attached to a Coleman family, which is a class \[ \kappa(\hf) \in H^1(\QQ, \frac{1}{X} V^c(\hf)^*(-\hj)). \] As we later discuss, this class depends on the choice of an auxiliary Dirichlet character $\chi$.
\end{enumerate}

Observe that $\hf$ is generically cuspidal, but is Eisenstein at one point; in particular, this allows us to consider the projection of $\frac{1}{X}V^c(\hf)^*$ onto the quotient \[ \frac{\frac{1}{X}V^c(\hf)^*}{V(\hf)^*} \cong \QQ_p(\tau^{-1})(1+r) \otimes \mathcal H_{\Gamma}(-\hj). \] Considering the projection in cohomology, we obtain a class \[ \kappa(f_{\beta}) \in H^1(\QQ, \QQ_p(\tau^{-1})(1+r) \otimes \mathcal H_{\Gamma}(-\hj)).\]

Letting $\Gamma^+ = \Gamma/\langle \pm 1 \rangle$, and writing $\hj_+$ for its universal character, we can see the class in the smaller subspace $H^1(\QQ, \QQ_p(\tau^{-1})(1+r) \otimes \mathcal H_{\Gamma}(-\hj_+))$, since the Iwasawa cohomology of the other half of the weight space vanishes. Our comparison of Beilinson--Kato classes and circular units will involve three different kinds of $p$-adic $L$-functions:
\begin{enumerate}
\item[(i)] The (one-variable) Kubota--Leopoldt $p$-adic $L$-function $L_p(\chi)$ attached to an even and non-trivial Dirichlet character $\chi$.
\item[(ii)] The (two-variable) Kitagawa--Mazur $p$-adic $L$-function $L_p(\hg,\chi)$ attached to a Coleman family $\hg$ and a Dirichlet character $\chi$.
\item[(iii)] The (one-variable) adjoint $p$-adic $L$-function attached to a Coleman family $\hg$, and denoted $L_p(\operatorname{Ad} \hg)$.
\end{enumerate}

We will see that the explicit reciprocity laws connect the class $c(\tau)$ with the $p$-adic $L$-function $L_p(\tau)$, and the class $\kappa(f_{\beta})$ with a ratio of $p$-adic $L$-functions which involve both $L_p(\hf)$ and $L_p(\operatorname{Ad} \hf)$, together with another auxiliary Kitagawa--Mazur $p$-adic $L$-function.

\subsection{Main results}

The first main result of this article may be seen as a direct comparison between the Kato Euler system and that of circular units. In order to state it as simple as possible in the introduction, we will assume that the leading term in the argument of Proposition \ref{prop:leading-term} corresponds to $n=0$. This means that we assume that $\kappa(f_{\beta})$ is equal to the class $\kappa^*(f_{\beta})$ introduced later in the text. For a more precise version of the result, see \ref{sec:def}. The Kato Euler system depends on the choice of an odd Dirichlet character $\chi$, and under the current sign assumptions, $L_p(\hf,\chi)$ vanishes at $(f_{\beta},r/2+1)$. Write $L_p'(f_{\beta},\chi)$ for the derivative of the two-variable $p$-adic $L$-function $L_p(\hf,\chi)$ restricted to the line $s=k/2$, and evaluated at $f_{\beta}$.

For the moment, we need an extra set of assumptions, which assert that the derivatives of the relevant $p$-adic $L$-functions do not vanish. The second part of the assumption may be relaxed, as we discuss in the body of the text.
\begin{assumption}\label{ass-2}
The adjoint $p$-adic $L$-function $L_p(\operatorname{Ad} \hf)$ has a simple zero at the critical Eisenstein point corresponding to $E_{r+2}^{\crit}(\psi,\tau)$. Moreover, $L_p'(f_{\beta},\chi)$ is non-zero.
\end{assumption}

It is known, by the results of Bella\"iche, \cite[\S9]{eigenbook}, that  $L_p(\operatorname{Ad} \hf)$ vanishes at $f_{\beta}$, but we do not know whether or not it presents a simple zero. The result we state also involves the {\it logarithmic distribution}: for a continuous character $\sigma : \ZZ_p^{\times} \rightarrow \CC_p$, the function $\frac{d^k \sigma}{dz^k} \cdot \frac{z^k}{\sigma(z)}$ is a constant, and $\log^{[k]} \in \CC_p$ is defined to be this constant. In addition, we write $\mathcal E_N$ for the two-variable function corresponding to the product of the bad Euler factors introduced in Definition \ref{def:bad-Euler}. Finally, let $\lambda$ denote the function defined over the Iwasawa algebra such that its specialization at a character $\sigma \colon t \mapsto t^s \xi(s)$ corresponding to an integer $s$ and a Dirichlet character $\xi$ of $p$-power conductor is $\xi^{-1}(N_2) \langle N_2 \rangle^{-s+1}$.

\begin{theorem}\label{theorem-main}
Suppose that $r$ is even and that the characters $\psi$ and $\tau$ are both even. Then, under Assumption \ref{ass-1} and Assumption \ref{ass-2}, the following equality holds in $H^1(\QQ, \QQ_p(\tau^{-1})(1+r) \otimes \mathcal H_{\Gamma}(-\hj))$: \[ \kappa(f_{\beta}) = \left( \frac{C \cdot \lambda \cdot \mathcal E_N \cdot L_p'(f_{\beta},\chi)}{L_p'(\operatorname{Ad} f_{\beta})} \right) \cdot L_p(\psi,\hj+1) \cdot c(\tau)(r), \] where $C$ is a constant which does not depend on $j$.
\end{theorem}

The constant $C$ arises during the interpolation of Ohta's canonical differentials, and it may be seen as the {\it less understood} piece of the picture; roughly speaking, it has to do with the interpolation of the Eichler--Shimura isomorphisms around critical Eisenstein points. Observe that the factors $L_p'(f_{\beta},\chi)$ and $L_p'(\operatorname{Ad} f_{\beta})$ only depend upon the weight, but we have opted to include them separately to emphasize this dependence. 
%The discussion around the leading term argument also lead us to assume that $L_p'(f_{\beta},\chi)$ is non-zero, although the situation where it vanishes can be handled in a similar way. 
This means that we have the equality \[ \kappa(f_{\beta}) = \tilde C \cdot \lambda \cdot \mathcal E_N \cdot L_p(\psi,\hj+1) \cdot c(\tau)(r), \] where $\tilde C$ is an arbitrary constant. In particular, it can be zero.

The main steps involved in the proof of the theorem can be summarized below.
\begin{enumerate}
\item[(1)] Deformation of the Beilinson--Kato elements in families passing through a critical Eisenstein point. The main result in this direction is stated as Theorem \ref{thm:bk-families}. Observe that the critical Eisenstein situation requires a careful study of the variation in Coleman families, in a setting where the arguments of \cite{BB22} and \cite{BB24} do not directly work.
\item[(2)] Use of the reciprocity laws for circular units and Beilinson--Kato elements (in families), which allow us to rephrase the statement of Theorem \ref{theorem-main} in terms of $p$-adic $L$-functions. This is further developed in Section \ref{sec:circular} and Section \ref{sec:bk-elts}.
\item[(3)] Interpolation of the canonical differentials $\eta_{\hf}$ and $\omega_{\hf}$ in families, together with a development of the explicit reciprocity laws around the {\it critical Eisenstein} point. The former is done in Section \ref{sec:dRcoh}, and the latter in Theorem \ref{theorem:rec-law-critical}.
\item[(4)] The Bloch--Kato conjecture for the trivial representation (twisted by a Dirichlet character and a power of the cyclotomic character). This shows that the classes $\kappa(f_{\beta})$ and $c(\tau)(r)$ belong to the same one-dimensional space.
\item[(5)] A factorization formula for the Hida--Rankin $p$-adic $L$-function when one of the forms is Eisenstein, following the work of Bertolini--Darmon \cite{bertolinidarmon14}. However, one must be cautious at this point when dealing with the different choices of periods.
\item[(6)] Bella\"iche--Dasgupta's factorization formula \cite{bellaichedasgupta15}, expressing the $p$-adic $L$-function of a critical Eisenstein series as the product of two Kubota--Leopoldt $p$-adic $L$-functions and a logarithmic term.
\end{enumerate}

{\bf The odd case.} If $r$ is even, $\tau$ and $\psi$ are both odd, and we further assume that $\tau(p) \neq 1$, one can directly claim that $\kappa(f_{\beta})=0$, since Iwasawa cohomology is zero along that half of the weight space, i.e., $H^1(\QQ, \QQ_p(\tau^{-1})(1+r) \otimes \mathcal H_{\Gamma}(-\hj_+)) = 0$. This means that the class $\kappa(\hf)$ lifts to an element in $H^1(\QQ, V(\hf)^*(-\hj_+))$. In this case, we can project $V(\hf)^*$ onto the cohomology of the quotient \[ \frac{V(\hf)^*}{V^c(\hf)^*} \cong \QQ_p(\psi^{-1}) \otimes \mathcal H_{\Gamma}(-\hj_+)), \] and this yields a class \[ \kappa(f_{\beta})_2 \in H^1(\QQ, \QQ_p(\psi^{-1}) \otimes \mathcal H_{\Gamma}(-\hj_+)). \] This class belongs to the one-dimensional space spanned by the circular unit Euler system $c(\psi)(-1)$. Unfortunately, our methods are not enough to get a good description of that class, and this will be developed in the sequel to this paper, extending the theory we have developed so far. This may be seen as an instance of {\it exceptional zero phenomenon}, where the vanishing, however, is not caused by a zero in an Euler factor.

\vskip 12pt

{\bf A conjectural generalization of \cite{bellaichedasgupta15}.} Assume now that $(-1)^r = -\psi(1)$. This is the situation, for instance, if $r$ is odd and the character $\tau$ is also odd; in that case, the class $\kappa(f_{\beta})$ is not expected to be zero, or at least not due to trivial sign reasons. If we analyze the $p$-adic $L$-functions of Theorem \ref{theorem-main}, we see that $L_p(f_{\beta},\chi)$ admits now a factorization formula (and, a priori, it is non-zero), but $L_p(\hf)$ is zero. 

In that case, we formulate a conjecture that may be seen as a natural analogue of Bella\"iche's secondary $p$-adic $L$-functions, and also as an extension of the work of Bella\"iche--Dasgupta. More precisely, we propose the following formula in Conjecture \ref{conj:fact}. Here, let $L_p'(\hf)$ stand for the weight derivative of $L_p(\hf)$; write $L_p'(\hf)(r,s)$ for the (one-variable) cyclotomic $p$-adic $L$-function corresponding to specializing $L_p(\hf)$ at $k=r$.

\begin{conjecture}
If $(-1)^s = -\psi(-1)$, the following factorization formula holds (up to multiplication by some explicit constant): \[ L_p'(f_{\beta})(s) = \lambda N_2^{-s} \log^{[r+1]} \cdot L_p(\psi,s) \cdot L_p(\tau,s-r-1), \] where $\lambda$ is the aforementioned function over the Iwasawa algebra and the derivative is taken along the weight direction.
\end{conjecture}

Later in the text we discuss theoretical evidence for the conjecture, whose proof seems beyond the scope of the methods developed in this note and in related works.

The organization of this note is as follows: we begin with an introduction of the Galois representations and periods involved in our comparison, with emphasis on the interpolation of the canonical differentials in a neighborhood of a critical Eisenstein point. Next, we fix our conventions for the systems of circular units and Kato elements, discussing the vanishing of the $p$-adic $L$-function of the adjoint and one of the Kitagawa--Mazur $p$-adic $L$-functions. This allows us to develop the study of the degeneration of Beilinson--Kato elements and carry out the comparison between both Euler systems. Section \ref{sec:other-cases} serves to discuss the limitation of our methods in the cases where the characters $\psi$ and $\tau$ are not both odd, and raises some interesting questions that we leave for forthcoming work; more specifically, there are two main unsolved issues: (a) the description of $\kappa(f_{\beta})_2$; (b) the extension of the factorization formula of \cite{bellaichedasgupta15}. We end up the note with an analogy with the simpler case of Beilinson--Flach elements, which is further developed in a sequel of this note, where we extend the methods of this article. 

In Appendix \ref{section: big Galois representations}, we discuss in detail the construction of rank-$2$ big Galois representations. The materials presented in this appendix should be well known to the experts, but we were unable to find a detailed reference. We also note that the appendix is written in a way that suggests a possible generalization to higher-rank groups.

%\subsection{The appendix to the paper}

\subsection{Relation with other works}

The starting point of this article was extending the results of \cite{LR1} to a setting where the same arguments were not available, since the comparisons developed in that work were based on the {\it strong} local conditions satisfied by the families of cohomology classes.

A similar idea in the framework of congruences among modular forms was explored in \cite{RiRo}. In that situation, the modulo $p$ congruences suggested the possibility of an extra vanishing in certain cases, which led us to analyze a second derivative with techniques that went beyond the explicit reciprocity laws. It would be interesting to study whether those results, mainly due to Sharifi and Fukaya--Kato, admit a counterpart in this scenario.

The idea of using the degeneration of Eisenstein series is not new in the theory of modular forms. For instance, in the setting of irregular weight one modular forms, when one takes the Eisenstein series of characters $(\chi_1,\chi_2)$, with $\chi_1(p)=\chi_2(p)$, there are three (ordinary) families passing through it: the families of Eisenstein series $\hE(\chi_1,\chi_2)$ and $\hE(\chi_2,\chi_1)$, but also a third family which is generically cuspidal. This idea is explored in the work of \cite{BDV} to prove Perrin-Riou's conjecture via a comparison between Beilinson--Kato systems and Heegner cycles.

Urban explored in \cite{Urban20}, with very different techniques, the phenomenon of Eisenstein congruences to study the construction of Euler systems. It would be interesting to understand the relationship between both approaches.

Finally, we expect to extend the study of these phenomena to settings beyond $\GL_2$, such as the Siegel case, where one can study endoscopy phenomena emerging in the theory of Saito--Kurokawa and Yoshida lifts.

\subsection{Acknowledgements}
We would like to thank Kazim B\"uy\"ukboduk for very helpful conversations and for his comments on an earlier version of the manuscript. The second author is especially grateful to David Loeffler, with whom he learned many of the techniques developed in this paper during the writing of the prequel \cite{LR1} and the subsequent discussions around the topic. The authors extend their gratitude to Samit Dasgupta, Alexandre Maksoud, Alice Pozzi, and Victor Rotger for informative conversations and correspondence around the topic of this article. J.P. and O.R. express their deepest gratitude to Ju-Feng Wu for writing the appendix to this paper.

During the preparation of this article, the authors were supported by PID2023-148398NA-I00, funded by MCIU/AEI/10.13039/501100011033/FEDER, UE.

\section{Galois representations and periods}

For this section, we recall some of the terminology and standard results around Galois representations that were explored in earlier work by Loeffler and the second author \cite{LR1}. The main novelty is the discussion of the interpolation of the differentials attached to a Coleman family passing through a critical Eisenstein point, which extends the results developed in loc.\,cit.

\subsection{Eisenstein series and Galois representations}

Let $p>2$ be a fixed prime number. Let $\psi, \tau$ be Dirichlet characters and $r \ge 0$ an integer with $\psi \tau(-1) = (-1)^r$. Let $N_1$ and $N_2$ denote the conductors of $\psi$ and $\tau$, respectively, and let $N = N_1 N_2$. Throughout this article, we assume that $p\nmid N$. Choose an embedding $\QQ(\psi, \tau) \hookrightarrow L$, where $\QQ(\psi, \tau) \subset \CC$ is the finite extension of $\QQ$ generated by the values of $\psi$ and $\tau$, and $L$ is a finite extension of $\Qp$.

We write $f = E_{r+2}(\psi, \tau)$ for the classical modular form of level $\Gamma_1(N)$ and weight $r+2$ given by 
\[
E_{r+2}(\psi, \tau) =  (*) + \sum_{n \ge 1} q^n \sum_{\substack{n = d_1 d_2 \\ (N_1, d_1) = (N_2, d_2) = 1}} \psi(d_1) \tau(d_2) d_2^{r+1},
\]
where $(*)$ is the constant defined e.g in \cite[Ch. 4]{diamondshurman}. In particular, it is equal to $\delta(\psi) L(\tau,1-k)/2$, where $\delta(\psi)=1$ if $\psi$ is trivial and 0 if not, and $L(\tau, \cdot)$ is the Dirichlet $L$-series attached to the character $\tau$. Note that we may regard the $q$-expansion coefficients $a_n(f)$ of $f = E_{r+2}(\psi, \tau)$ as elements of $L$. 

The $p$-th Hecke polynomial of $f$ is \[ x^2 - (\psi(p)+\tau(p)p^{r+1})x + \psi \tau(p)p^{r+1} = (x-\psi(p))(x-\tau(p)p^{r+1}). \] Then, associated to the Eisenstein series $E_{r+2}(\psi,\tau)$, there are two different $p$-stablisations: one ordinary and one of critical slope, with eigenvalues $\alpha := \psi(p)$ and $\beta := p^{r+1} \tau(p)$, respectively. We denote these eigenforms by $f_{\alpha}$ and $f_{\beta}$. In terms of $q$-expansions, \[ f_{\beta}(q) = f(q) - \alpha f(q^p), \] and similarly for $f_{\alpha}(q)$. 

\begin{definition}
We say $f_{\beta}$ is {\it non-critical} if the generalised eigenspace of overconvergent modular forms $M_{r+2}^{\dag}(\Gamma)_{(T=f_{\beta})}$ associated to $f_{\beta}$ is 1-dimensional.
\end{definition}

In \cite[Rk. 1.5]{bellaichedasgupta15} the authors discuss how to reinterpret the criticality condition in terms of the non-vanishing of a $p$-adic $L$-value. The next definition plays a crucial role for establishing the properties of the Galois representations.

\begin{definition}
We say the Eisenstein series $f = E_{r+2}(\psi, \tau)$ is \emph{$p$-decent} if either $r > 0$ or $r = 0$, and for every prime $\ell \mid Np$, either the conductor of $\psi / \tau$ is divisible by $\ell$, or $(\psi/\tau)(\ell) \ne 1$. 
\end{definition}

We now fix notation for Galois characters.
\begin{itemize}
\item Let $\ell$ be a prime. We write $\Frob_\ell$ for an arithmetic Frobenius element at $\ell$ in $\operatorname{Gal}(\Qb / \QQ)$; this depends on the choice of an embedding of $\Qb$ into $\Qb_{\ell}$, and is well defined modulo the inertia subgroup for this embedding.

\item The $p$-adic cyclotomic character $\operatorname{Gal}(\Qb / \QQ) \to \Qp^\times$ is denoted by $\eps$, and $\eps(\Frob_\ell) = \ell$ for $\ell \ne p$.

\item If $\chi$ is a Dirichlet character, we interpret $\chi$ as a character $\Gal(\Qb / \QQ) \to \CC^\times$ unramified at all primes $\ell$ not dividing the conductor, and mapping $\Frob_\ell^{-1}$ to $\chi(\ell)$ for all those primes $\ell$.
\end{itemize}

The following result, due to Soul\'e \cite{soule-81}, summarizes the different possibilities that we have for Galois representations whose traces are prescribed by the Eisenstein eigenvalues of the Hecke operators. We state it here with the conventions of \cite{bellaichechenevier06}.

\begin{theorem}[Soul\'e]
\label{thm:soule}
If $f$ is a $p$-decent Eisenstein series, there are exactly three isomorphism classes of continuous Galois representations $\rho: G_\QQ \to \GL_2(L)$ which are unramified at primes $\ell \nmid Np$ and satisfy $\operatorname{tr} \rho(\operatorname{Frob}_\ell^{-1}) = a_\ell(f)$. These are as follows:
\begin{enumerate}
\item The semisimple representation $\psi \oplus \tau \eps^{-1-r}$.

\item Exactly one non-split representation having $\tau \eps^{-1-r}$ as a subrepresentation. This representation splits locally at $\ell$ for every $\ell \ne p$, but does not split at $p$, and is not crystalline (or even de Rham).

\item Exactly one non-split representation having $\psi$ as a subrepresentation. This representation splits locally at $\ell$ for every $\ell \ne p$, and is crystalline at $p$.
\end{enumerate}
\end{theorem}

\begin{proof}
This follows from cases of the Bloch--Kato conjecture due to Soul\'e; see \cite[\S 5.1]{bellaichechenevier06} for the statement in this form.
\end{proof}

If $f$ is not $p$-decent, there will be extra representations in case (3), but we can repair the statement by only considering representations which are assumed to be unramified (resp.~crystalline) at each prime $\ell \ne p$ (resp.~at $\ell = p$) where the hypothesis of being decent fails.

\subsection{Families of representations}\label{subsec:representations}

Proceeding as in \cite[\S A.4,A.5]{LR1}, we can naturally associate four Galois representations to $f_\beta$: $V(f_\beta)$, $V(f_\beta)^*$, $V^c(f_\beta)$, and $V^c(f_\beta)^*$. The superscript ${}^c$ always stands for compact support. All of them are two dimensional, with natural one dimensional invariant subspaces. In particular, all of them are non-split extensions, and $V(f_\beta)^*$ is isomorphic to the dual of $V(f_\beta)$ (with both being de Rham at $p$), while $V^c(f_\beta)^*$ is isomorphic to the dual of $V^c(f_\beta)$.

For the representation $V^c(f_{\beta})$, there is a short exact sequence of $L$-vector spaces \[ 0 \longrightarrow L(\tau \eps^{-1-r}) \longrightarrow V^c(f_\beta) \longrightarrow L(\psi) \longrightarrow 0, \] and from here we get another short exact sequence of $L$-linear Galois representations for $V(f_{\beta})$: \[ 0 \longrightarrow L(\psi) \longrightarrow V(f_\beta) \longrightarrow L(\tau \eps^{-1-r}) \longrightarrow 0. \]
Similarly, for $V(f_{\beta})^*$, there is a short exact sequence
\[ 0 \longrightarrow L(\tau^{-1} \eps^{1 + r}) \longrightarrow V(f_\beta)^* \longrightarrow L(\psi^{-1}) \longrightarrow 0, \]
while for the compactly supported cohomology we have \[ 0 \longrightarrow L(\psi^{-1}) \longrightarrow V^c(f_\beta)^* \longrightarrow L(\tau^{-1} \eps^{1 + r}) \longrightarrow 0. \]
There is a natural map $V^c(f_\beta)^* \to V(f_\beta)^*$ whose image is the $\tau^{-1}\eps^{1+r}$ subrepresentation of the latter; see e.g. the discussion around \cite[Rk. 5.10]{bellaichedasgupta15}.

Since the objective is considering families of modular forms passing through $f_{\beta}$, we need to make an assumption about the behavior of the eigencurve at that point. More precisely, with this purpose in mind, we assume for the rest of the text the following condition.

\begin{assumption}
The Coleman--Mazur--Buzzard cuspidal eigencurve $\mathcal C^0$ of tame level $\Gamma_1(N)$ is smooth at $f_{\beta}$ and locally \'etale over the weight space.     %\textcolor{red}{(tame level $N$ means $\Gamma_1(N)$ or $\Gamma_0(N)$ or $\Gamma(N)$?)}
\end{assumption}

This assumption is automatic when $f$ is {\it decent} and {\it non-critical}; as discussed in \cite[\S A]{LR1}, these are fairly general conditions. Moreover, this means that may take an affinoid open $U = \mathrm{Spa}(A_U, A_U^{\circ})$ around $f_{\beta}$ mapping isomorphically to its image in the weight space.

Let $\hf$ be the universal eigenform over $U$, which is an overconvergent cuspidal eigenform with coefficients in $A_U$. We also write $\beta_{\hf} \in A_U^{\times}$ for the Hecke $U(p)$-eigenvalue of $\hf$. There is also a canonical Galois pseudo-character $t_{\hf} \colon G_{\QQ} \rightarrow A_U$ satisfying \[ t_{\hf}(\Frob_{\ell}^{-1}) = a_{\ell}(\hf) \] for all good primes $\ell$.

According to Proposition \ref{prop: big Galois representations} and Remark \ref{remark: dual construction}, there exist Galois representations $V(\hf)$, $V^c(\hf)$ (resp.,  $V(\hf)^*$, $V^c(\hf)^*$), free of rank-$2$ over $A_U$ whose traces equal to $t_{\hf}$ (resp., $t_{\hf}^*$). Moreover, by Corollary \ref{cor: injection of big Galois representations}, we have injections of Galois representations \[
    V^c(\hf) \hookrightarrow V(\hf) \quad (\text{resp., }V^c(\hf)^*\hookrightarrow V(\hf)^*).
\] 
We refer the reader to Appendix \ref{section: big Galois representations} for the detailed construction of these representations.

Let $Y$ be the modular curve of level $\Gamma_1(N) \cap \Gamma_0(p)$ over $\QQ$ and let $\overline{Y}$ be the base change of $Y$ to $\overline{\QQ}$. Let $V = \mathrm{Spa}(A_V, A_V^{\circ})$ be the image of $U$ under the weight map (note that we shrink $U$, if necessary, so that $U$ is isomorphic to $V$) and let $\lambda_V: \ZZ_p^\times \rightarrow A_V^\times$ be the universal weight of $V$. Let $s \geq 1+s_V$ (cf. Remark \ref{remark: some remarks on the notion of weights}) and consider the {\'e}tale cohomology $H^1_{\mathrm{\acute{e}t}}(\overline{Y}, \mathscr{D}_{\lambda_V}^s)$ (cf. Definition \ref{def: etale cohomology for affinoid weights}). From the construction of $V(\hf)$, one sees that it is a Hecke-equivariant direct summand of $H^1_{\mathrm{\acute{e}t}}(\overline{Y}, \mathscr{D}_{\lambda_V}^s)$ (and is independent to $s$). In particular, there exists a natural Hecke-equivariant projection \[
    \operatorname{Pr}_{\hf}: H^1_{\mathrm{\acute{e}t}}(\overline{Y}, \mathscr{D}_{\lambda_V}^s) \rightarrow V(\hf).
\]

By \cite[Thm. 1]{bellaichechenevier06}, $t_{\hf}$ has maximal reducibility ideal. Then, it follows that the fiber of $V(\hf)$ at $k=r$, which is exactly $V(f_\beta)$, must be a non-split extension. The same applies to the other three modules, showing that all four of $V(f_\beta), V^c(f_\beta), V(f_\beta)^*$ and $V^c(f_\beta)^*$ are non-split extensions in the setting of Soul\'e's theorem.

Since the map $V^c(f_\beta) \to V(f_\beta)$ is not the zero map, $V^c(\hf)$ is not contained in $\mathfrak{m} \cdot V(\hf)$. Thus we may find a basis $(e_1, e_2)$ of $V(\hf)$, and an integer $r \ge 1$, such that $(e_1, X^r e_2)$ is a basis of $V^c(\hf)$, where $X$ is a uniformizer of $\mathfrak{m}$. We have $X V(\hf) \subset V^c(\hf)$, and this means there is a chain of inclusions
\[ \dots \supset \tfrac{1}{X} V^c(\hf) \supset V(\hf) \supset V^c(\hf) \supset X V(\hf) \supset \dots \]
with all of the successive quotients l-dimensional over $L$, and alternately equal to either $\psi$ or $\tau \varepsilon^{-1-r}$ as Galois modules. Similarly, we have a chain
\begin{equation}\label{chain}
\dots \supset \tfrac{1}{X} V^c(\hf)^* \supset V(\hf)^* \supset V^c(\hf)^* \supset X V(\hf)^* \supset \dots
\end{equation}
with quotients alternately isomorphic to either $\psi^{-1}$ or $\tau^{-1} \varepsilon^{1+r}$.

\subsection{Crystalline periods}\label{sec:dRcoh}

The aim of this section is describing the canonical periods arising in the reciprocity laws, following the seminal developments of Ohta \cite{ohta00} and Kings--Loeffler--Zerbes \cite[\S10]{KLZ17}. This is one of the subtlest and most technical parts of the paper, since we do not have a good control over the interpolation of the canonical differential.

We begin by establishing some notation. Fix an embedding $\Qb \into \Qpb$, that allow us to see $\Gal(\Qpb/\Qp)$ as a subgroup of $\Gal(\Qb/\QQ)$. We write $\cR$ for the Robba ring over $\Qp$, $U = \mathrm{Spa}(A_U, A_U^{\circ})$ be an affinoid open subset of $\mathcal{C}^0$, and let $\cR_U = \cR \mathop{\hat\otimes}_{\QQ_p} A_U$. A $(\varphi, \Gamma)$-module is endowed with a Frobenius action (denoted by $\varphi$) and a Galois action, for which we take a topological generator $\gamma$ of $\Gal(\Qpb/\Qp)$. As usual, let $t \in \cR$ be the period for the cyclotomic character, so $\varphi(t) = pt$ and $\gamma(t) = \varepsilon(\gamma) t$. For $\delta \in A_U^\times$, let $\cR_U(\delta)$ be the rank-1 $(\varphi, \Gamma)$-module over $\cR_U$ generated by an element $e$ which is $\Gamma$-invariant and satisfies $\varphi(e) = \delta e$.\footnote{ We warn the reader that this is not the standard notation for rank-1 $(\varphi, \Gamma)$-modules. We use this notation to avoid introducing more technical terminology. } Finally, when $D$ is a $(\varphi, \Gamma)$-module over $\cR$ (or $\cR_U$), write $\Dcris(D) = D[1/t]^{\Gamma}$, with its filtration $\Fil^n \Dcris(D) = (t^n D)^{\Gamma}$. Let $D(\hf)^* = \mathbf{D}^\dag_\rig(V(\hf)^*)$, and similarly $D^c(\hf)^*$.

The space of homomorphisms \[ \cR_U\left(\frac{\beta_\hf}{\psi(p) \tau(p)}\right)(1 + \hk) \to D(\hf)^*\] is a finitely-generated $A_U$-module, by the main theorem of \cite{KPX}. In particular, it is free of rank 1, and there is a Zariski-dense set of $x \in A_U$ where the fiber is 1-dimensional. Let $\cF^+ D(\hf)^*$ be the image of a generator of this map, and $\cF^-$ the quotient, so that we have a short exact sequence
\[ 0 \to \cF^+ D(\hf)^* \to D(\hf)^* \to \cF^- D(\hf)^* \to 0.\]

If $f_{\beta}$ is non-critical, the submodule $\mathcal F^+ D(f_{\beta})^*$ is saturated, and $\mathcal F^- D(\hf)^*$ is free of rank 1 over $\mathcal R_U$. Further, $\mathcal F^{\pm} D(\hf)^*$ define a triangulation of $D(\hf)^*$ over $U$.  

Following the previous discussion, we can make the following definition.

\begin{definition}
Let $b_{\hf}^+$ be any fixed (and non-canonical) isomorphism \[ b_{\hf}^+: \Dcris\left(\cF^+ D(\hf)^*(-1-\hk) \right) \cong A_U. \]
\end{definition}

If $g$ is a non-critical classical specialization of $\hf$, with $g$ of weight $k+2$ for some $k \ne r$ then we have a \emph{canonical} isomorphism between the fibers of the above modules at $k$, given by the image modulo $\cF^-$ of the class in $\Fil^{k+1} \Dcris\left(V(g) \right)$ of the differential form $\omega_g$ associated to $g$. Note that here we are using the comparison isomorphism between de Rham and \'etale cohomology crucially.

So, for each such $g$, there must exist a non-zero constant $c_g \in L^\times$ such that $b_{\hf}^+$ specializes to $c_g \omega_g$.

\begin{proposition}
There exists some scalar $c_{f_\beta} \in L^\times$ such that
\[ b_{f_\beta}^+ = c_{f_\beta} \eta_{f_\beta}. \]
\end{proposition}

\begin{proof}
At $X = 0$, we have an isomorphism
\[ \Dcris(\cF^+ D(f_\beta)^*) \cong \Dcris(V(f_\beta)^*_{\mathrm{quo}}), \]
and we have a duality between $V(f_\beta)^*_{\mathrm{quo}}$ and $V^c(f_\beta)_{\mathrm{quo}}$, where the subscript $\mathrm{quo}$ stands for the quotient in the corresponding filtrations, as defined in \S\ref{subsec:representations}. This allows us to see $b_{f_\beta}^+$ as a basis of the space \[ \Dcris(V^c(f_\beta)_{\mathrm{quo}}) \cong H^1_{\mathrm{dR}, \mathrm{c}}(Y, \mathscr{V}_r)[\mathbb{T} = f_\beta] = H^1(X, \omega^{-r}(-\mathrm{cusps}))[\mathbb{T} = f_\beta],\footnote{ Here, we use the same notation as in Assumption \ref{assump: multiplicity-one}.} \] where the latter stands for the higher coherent cohomology of the compactification $X$ of $Y$ (cf. \cite[\S A3]{LR1}).
Therefore, there exists some scalar $c_{f_\beta} \in L^\times$ such that
\[ b_{f_\beta}^+ = c_{f_\beta} \eta_{f_\beta}, \] although a priori we do not have any control over that constant.
\end{proof}

\begin{remark}
It is interesting to note that the element $b_{\hf}^+$ ``generically'' interpolates the $\Fil^1$ vectors $\omega_g$ for specializations $g$ of weight $\ne r$, but at the bad weight $k = r$, it interpolates the $\Fil^0$ vector $\eta_{f_\beta}$ instead.
\end{remark}

Proceeding as before, we can make the following definition.

\begin{definition}
Let $a_{\hf}^-$ be any fixed (and non-canonical) isomorphism  \[ a_{\hf}^-: \Dcris\left(\cF^- D(\hf)^* \right) \cong A_U. \]
\end{definition}

In particular, if $g$ is a non-critical classical specialization of $\hf$, then $a_{\hf}^-$ specializes to $d_g \eta_g$, where $d_g$ is some arbitrary non-zero constant.

\begin{proposition}
There exists some scalar $d_{f_\beta} \in L^\times$ such that
\[ a_{f_\beta}^- = d_{f_\beta} \omega_{f_\beta}. \]
\end{proposition}

\begin{proof}
At $X=0$, we have an isomorphism \[ \Dcris(\cF^- D(f_\beta)^*(-1-r)) \cong \Dcris(V(f_\beta)^*_{\mathrm{sub}}). \] Taking into account the dualities, and proceeding as before, we may claim that there exists a scalar $d_{f_{\beta}} \in L^{\times}$ such that $a_{f_{\beta}}^- = d_{f_{\beta}} \omega_{f_{\beta}}$.
\end{proof}

\begin{remark}
The isomorphism we have used in the proof of the previous proposition is an isomorphism of the underlying $\varphi$-modules. Following the discussion after \cite[Prop. A8.1]{LR1}, there is a shift in the filtration degree by $r+1$.
\end{remark}

Observe that while in \cite{LR1} we only needed the interpolation of $\omega_{\hg}$, we now incorporate in the picture the differential $\eta_{\hg}$, that will play a role in the situations where the critical family is dominant. See e.g. \cite[\S10]{KLZ17} for a discussion of the general theory of the canonical differentials and their interpolation in families.

\section{Circular units}\label{sec:circular}

The main aim of this work is exploring the degeneration of the Euler system of Beilinson--Kato classes when it passes through a critical Eisenstein series. In those cases, we expect to recover the system of {\it circular units}. In order to properly introduce it, let $\chi$ be an even, primitive, and non-trivial Dirichlet character of conductor $N$ (with $p \nmid N$), taking values in $L$. Write $\mathcal O_{\chi,p}$ for the ring of integers of an extension of $\QQ_p$ containing the values of $\chi$.

As in the introduction, let $\mathcal H_{\Gamma}$ be the distribution algebra of $\Gamma \cong \ZZ_p^{\times}$ (with $L$-coefficients); that is, $\mathcal{H}_{\Gamma}$ is the global section of the rigid analytic space associated with $\mathrm{Spf}\, \mathcal{O}_L[\![\Gamma]\!]$. %This is often denoted as $\mathcal O(\mathcal W)$, where $\mathcal W = \mathrm{Hom}_{\mathrm{cont}}(\mathbf{Z}_p^{\times}, \mathbf{C}_p^{\times})$, but we have opted to keep the notations of our earlier work on Eisenstein degeneration. 
For our further use, let $\hj \colon \Gamma \hookrightarrow \mathcal H_{\Gamma}^{\times}$ be the universal character, which we regard as a character of the Galois group by composition with the cyclotomic character. The Kubota--Leopoldt $p$-adic $L$-function attached to $\chi$, denoted by $L_p(\chi)$, is an element of $\cH_{\Gamma}$ which interpolates special values of the Dirichlet $L$-function of $\chi$. We refer \cite[\S3]{Dasgupta-factorization} for a more detailed treatment of the properties of this $p$-adic $L$-function with our current conventions, which slightly differ from \cite{colmez00}.

\begin{proposition}\label{prop:kl}
There exists a unique element $L_p(\chi) \in \cH_{\Gamma}$ satisfying the following interpolation property: 
\[ L_p(\chi)(j) = \begin{cases}
(1- p^{j-1} \chi(p)^{-1}) \frac{2N^j (j-1)!}{(-2\pi i)^j \tau(\chi)} L(\chi, j) & \text{$j \ge 2$ even},\\
(1 - p^{-j} \chi(p)) L(\chi, j) & \text{$j \le -1$ odd}
\end{cases}, \]
where $\tau(\chi) = \sum_{a \in (\ZZ/N)^\times} \chi(a)\zeta_N^a$ is the Gauss sum. 
\end{proposition}

\begin{remark}
We follow a post of David Loeffler in MathOverflow to explain how this interpolation property is consistent with other conventions. We are grateful to him for a nice discussion on the conventions around that point.

The element $L_p(\chi)$ belongs to the Iwasawa algebra $\mathcal O_{\chi,p}[\![\ZZ_p^{\times}]\!]$, which is a product of $(p-1)$ subrings, indexed by $\ZZ/(p-1)\ZZ$, each of them isomorphic to $\mathcal O_{\chi,p}[\![T]\!]$. 

Then, for any $t \in \ZZ/(p-1)\ZZ$ with $(-1)^t = \chi(-1)$, there exists a continuous function $L_{p,t}(\chi) \colon \ZZ_p \rightarrow \mathbf{C}_p$ such that, for all integers $k \geq 0$ with $k \equiv t \pmod{p-1}$, we have $L_{p,t}(\chi,k) = (1-\chi(p)p^{k-1}) \cdot L(\chi,1-k)$. For the subrings corresponding to the sign condition $(-1)^t = -\chi(-1)$, we define the corresponding $L_{p,t}(\chi)$ using the interpolation property along the positive integers, thus having $(p-1)$ different $p$-adic $L$-function which can be packaged in a single element $L_p(\chi)$ of $\mathcal O_{\chi,p}[\![\ZZ_p^{\times}]\!]$.
\end{remark}

With these conventions, the Kubota--Leopoldt $p$-adic $L$-function satisfies a functional equation, which is discussed in \cite[\S3.2]{Dasgupta-factorization}.
However, although the Kubota--Leopoldt $p$-adic $L$-function may be defined along any of the $(p-1)$ discs of the weight space, regardless of any parity constraint, the cohomology classes can only be defined over half of them.

\begin{definition}
Let $\chi$ be an even, primitive, and non-trivial Dirichlet character of conductor $N$. Let $H$ be the subfield of $\QQ(\zeta_N)$ cut out by $\ker(\chi)$, and let
\[ u_{\chi} = \prod_{a=1}^{N-1} \chi(a) \otimes (1-\zeta^a) \in (\mathcal O_L \otimes \mathcal O_H^{\times})^{\chi}. \]
\end{definition}

This kind of elements, often called {\it circular units}, are the essential piece for constructing the easiest instance of Euler systems, that of circular units. Consider the module
\[ V(\chi)^*(-\hj) := \mathcal O_{\chi,p}(\chi^{-1})\otimes \mathcal H_{\Gamma}(-\hj). \]
Here, as usual, we identify $\chi$ with a character of the Galois group. It is characterized by the property that for any integer $j$, we recover $V(\chi)^*(j) = \mathcal O_{\chi,p}(\chi^{-1})(-j)$ by specializing at $j$ (viewed as a ring homomorphism $\mathcal H_{\Gamma} \to \Qp$.

Modifying the previous units as described in \cite{perrinriou94b} (see also \cite[Section 1.1]{BCDDPR}), we obtain the so-called Euler system of circular units, which is an element
\[ c(\chi) \in H^1(\QQ, V(\chi)^*(-\hj)) = \varprojlim_n H^1(\QQ(\mu_{p^n}), V(\chi)^*(1)), \]
whose image in $H^1(\QQ, V(\chi)^*(1))$ is the image of $(1 - \chi(p)) \cdot u_\chi$ under the Kummer map. Note that the Iwasawa cohomology of $V(\chi)^*(1)$ is actually zero over half of the weight space; so we may obtain a slightly more refined statement by writing $\Gamma^+ = \Gamma/\langle \pm 1 \rangle$, and $\hj_+$ for its universal character, and considering $\kappa(\chi)$ as an element of the slightly smaller space $H^1(\QQ, V_p(\chi)^*(1 - \hj_+))$.

\begin{remark}
We warn the reader that our convention for identifying Dirichlet characters with Galois characters is the inverse of that of \cite{BCDDPR}, so our $c(\chi)$ would be $\kappa(\chi^{-1})$ in their notation.
\end{remark}

As a piece of notation for our further use, we write $c(\chi)(r) \in H^1(\QQ, \QQ_p(\tau^{-1})(1+r) \otimes \mathcal H_{\Gamma}(-\hj))$ for the Iwasawa cohomology class obtained from $c(\chi)$ after twisting by the power of the cyclotomic character.

The reciprocity law, as established in \cite{perrinriou94b}, asserts the following. Here, $\operatorname{Col}_{\chi}$ denotes the Coleman map, which interpolates either the Bloch--Kato logarithm or the dual exponential map according to the value of $j$. We refer the reader to any of the aforementioned references for more details.

\begin{proposition}
There is an equality \[ \operatorname{Col}_{\chi}(\loc_p(c(\chi))) = L_p(\chi)\big|_{\Gamma^+}, \] where $\loc_p$ means localization at $p$ and $\operatorname{Col}_{\chi}$ is the Coleman homomorphism for $V_p(\chi)^*(1)$.
\end{proposition}

\begin{proof}
This directly follows from \cite{perrinriou94b}, adapting her conventions to our setting.
\end{proof}

We later discuss the case of an odd Dirichlet character $\chi$. In this setting, circular units do not exist as described so far, but it is still possible to make sense of an element $c(\chi) \in H^1(\QQ, V_p(\chi)^*(1 - \hj_-))$, which is obtained through the interpolation of $u_{\chi \xi}$, where $\xi$ is an odd character of $p$-power conductor; alternatively, we may modify Proposition \ref{prop:kl} to include this case by considering not just the parity of $j$, but $(-1)^j \chi(-1)$ instead.

\section{Beilinson--Kato elements}\label{sec:bk-elts}

The Kato Euler system (or Beilinson--Kato Euler system) is one of the most significant constructions in the study of $p$-adic methods in Number Theory. It has enabled the proof of new cases of the Bloch--Kato conjecture and provided a general approach to establish one divisibility in the Iwasawa main conjecture. More concretely, the Kato Euler system of \cite{kato04} can be attached to a general Coleman family $\hf$, as discussed e.g.~ by Hansen in \cite{Hansen-Iwasawa} or by Benois and B\"uy\"ukboduk in \cite{BB22} and \cite{BB24}; for the the case of Hida families, the theory is known since the work of Ochiai \cite{ochiai03}. In this section, we begin by recalling the (two-variable) $p$-adic $L$-function of Kitagawa and Mazur, and later discuss the reciprocity law, emphasizing the role of the (one-variable) adjoint $p$-adic $L$-function. This discussion, and in particular the proof of the reciprocity law, does not hold at the critical Eisenstein point, so we later need to carry out certain modifications to adapt it.

\subsection{The Kitagawa--Mazur $p$-adic $L$-function}\label{sec:KM}

Let $f_k$ be the newform associated to the weight $k+2$ specialization of $\hf$, for some $k \geq 0$. Assume that $f_k$ does not correspond to a $p$-stabilization of an Eisenstein series, and let $\mathbf{Q}(f_k)$ the field of coefficients of $f_k$. Then, for each sign $\pm$, the eigenspace in Betti cohomology of $Y_1(N)$ on which the Hecke operators act by the eigensystem of $f_k$ is one-dimensional, and we choose bases $\gamma_f^{\pm}$ of theses spaces. These determine complex periods $\Omega_{f_k}^{\pm} \in \CC^{\times}$.

With the previous notations, let $U = \mathrm{Spa}(A_U, A_U^{\circ}) \subset \mathcal{C}^0$ be the affinoid open over which the family $\hf$ is defined. Then we have an $A_U$-module $V(\hf)^*$, free of rank 2. We choose bases $\gamma^{\pm}_{\hf}$ of the complex-conjugation eigenspaces $V(\hf)^{(c = \pm 1)}$, which are free of rank one over $A_U$. As pointed out in \cite[\S B2.1]{LR1}, we need to extend the definition of $\Omega_{f_k}^{\pm}$ in such a way that they lie in the space $(L \otimes_{\QQ(f_k)} \CC)^{\times}$. %\textcolor{red}{(JF: Previously, $U$ is an affinoid open neighbourhood defining $\mathbf{f}$, which is different from an affinoid open neighbourhood in the weight space, although they are isomorphic. Better to be consistent?)}

\begin{definition}
Let \[ L_p(\hf) \in A_U \hat{\otimes}_{\QQ_p} \cH_\Gamma \] be the two-variable Kitagawa--Mazur $p$-adic $L$-function attached to the family $\hf$ and the periods $\gamma_{\hf}^{\pm}$, interpolating the critical $L$-values of all classical, weight $\ge 2$ specializations of $\hf$.
\end{definition}

More precisely, the value at $(k, j)$, with $0 \le j \le k-2$, interpolates the $L$-value $L(f_k, 1 + j)$ up to appropriate factors, where $f$ is the weight $k$ specialization. See e.g. \cite[\S4]{fukayakato12} or \cite{BB22} for the precise form of the interpolation property, which involves different kinds of factors.

This construction was considered in generality by Bella\"iche in \cite{bellaiche11a}, using modular symbols; note that there are cases where the interpolation property does not completely characterize the $p$-adic $L$-function, which is not the case in the ordinary situation.

\subsection{The adjoint $p$-adic $L$-function}

For this section, we mainly follow the works of Hida \cite{hida-adjoint} and Bella\"iche \cite[\S9] {eigenbook}, who constructed an adjoint $p$-adic $L$-function, using Kim's scalar product on overconvergent modular symbols. This was later generalized to a more general setting in the recent works of Balasubramanyam, Bergdall, and Longo \cite{BBL20}, or Lee and Wu \cite{LW-Bianchi}. See also \cite{maksoud23} for a detailed treatment of the topic.

Before establishing the reciprocity law for Beilinson--Kato elements, we need to present another $p$-adic $L$-function attached to the Coleman family $\hf$. For any classical and cuspidal specialization $f$ of $\hf$, we have plus and minus periods $\Omega^{\pm}_{\hf}$. The ratio
\[ L^{\mathrm{alg}}(\operatorname{Ad} f, 1) \coloneqq \frac{-2^{k-1} i \pi^2 \langle f, f \rangle}{\Omega^+_f \Omega^-_f} \]
is in $\QQ(f)^\times$. 

If we choose bases $\gamma_{\hf}^{\pm}$ over the family as above, and use the periods for each $f_k$ determined by these, then a construction due to Hida \cite{hida-adjoint}, and more generally to Bella\"iche \cite{eigenbook}, gives a $p$-adic adjoint $L$-function $L_p(\operatorname{Ad} \hf)$. As a further piece of notation, we say that a family $\hf$ is $p$-distinguished if the semisimplification of the mod-$p$ Galois representation attached to $\hf$, $\bar{\rho}_{\hf}|_{G_{\QQ_p}}$, is the direct sum of two distinct characters.

\begin{proposition}
There exists a one-variable adjoint $p$-adic $L$-function \[ L_p(\operatorname{Ad} \hf) \in A_U \]interpolating the ratios \[ L_p(\operatorname{Ad} \hf)(k) = p^{k-1}\alpha_{k}(p-1) \left(1 - \tfrac{\beta_{k}}{\alpha_{k}}\right)\left(1 - \tfrac{\beta_{k}}{p\alpha_{k}}\right) L^{\mathrm{alg}}(\operatorname{Ad} f_k, 1), \]
for all $k \in U \cap \ZZ_{\ge 0}$ such that the specialization $f_k$ is the $p$-stabilization of a level $N$ cuspidal newform. If $\hf$ is $p$-distinguished, then the congruence ideal of $\hf$ is principal, and this ideal is generated by $L_p(\operatorname{Ad} \hf)$. 
\end{proposition}

As pointed out in \cite{maksoud23}, Bella\"iche's interpolation formula is only given at crystalline points of trivial nebentype, while in Thm. 2.3.7 of loc.\,cit. the author gives a much more general result which, however, is only valid in the setting of ordinary Hida families.

The adjoint $p$-adic $L$-function vanishes at the ramification points of the eigencurve. For example, this is the case for the weight one RM theta series. At a critical Eisenstein point, this $p$-adic $L$-function also vanishes. More concretely, we recall the following result from \cite[\S9]{eigenbook}.

\begin{theorem}
If $x \in U$ is a point that either has non-integral weight, or is cuspidal classical, then the adjoint $p$-adic $L$-function has a zero if and only if the weight map is not \'etale at $x$. Otherwise, it has a zero.    
\end{theorem}

\begin{proof}
This follows from the results of \cite[\S9]{eigenbook}, based on Kim's arguments \cite{kim06}.
\end{proof}

In particular, the adjoint $p$-adic $L$-function vanishes at the critical Eisenstein points, although the ramification index is $e=1$; this is in contrast with the cuspidal setting, where it could be understood as a measure of ramification. In particular, \cite{BB22} and \cite{BB24} discuss the relationship between the ramification index of the cuspidal eigencurve and the order of vanishing of the adjoint $p$-adic $L$-function.

\subsection{Hida--Rankin vs Kitagawa--Mazur}\label{sec:HR}

The Beilinson--Kato Euler system is associated with the Galois representation of a single modular form or family. However, the reciprocity law we aim to establish does not relate it to the Kitagawa--Mazur $p$-adic $L$-function directly, but rather to the product of two such functions, with dependence on the choice of an auxiliary Dirichlet character. Furthermore, this product is not exactly composed of two Kitagawa--Mazur functions; instead, it corresponds to a Hida--Rankin $p$-adic $L$-function, which agrees with the former up to the choice of periods (represented by the adjoint $L$-function). Specifically, for a modular form $f \in S_k(N, \chi)$, there are two possibilities for the choice of periods: the Petersson product $\langle f, f \rangle$, and the product $\Omega_f^+ \Omega_f^-$. As we have already discussed, the discrepancy between these choices is measured by the $p$-adic $L$-function of the adjoint representation.

The following result establishes how the Hida--Rankin $p$-adic $L$-function $L_p(\hf, \hg)$ attached to two families (either Hida or Coleman families) can be expressed as the product of two Kitagawa--Mazur $p$-adic $L$-functions. In particular, note that the $p$-adic $L$-function $L_p(\hf, \hg)$ is a three-variable $p$-adic $L$-function, depending on the two weights and the cyclotomic variable.

\begin{proposition}\label{prop:factor-hr}
Let $\hg$ be the (ordinary) Eisenstein family of characters $(\chi_1,\chi_2)$. Then, the following factorization formula holds: \[ L_p(\hf,\hg)(k,\ell,s) \times L_p(\operatorname{Ad} \hf)(k) = \mathfrak f(k) \times L_p(\hf,\chi_1)(k,s) \times L_p(\hf,\chi_2)(k,s-\ell+1), \] where $\mathfrak f(k)$ is an explicit factor depending just on $k$.
\end{proposition}

\begin{proof}
This follows by a direct comparison of the critical values, as in \cite[Thm. 3.4]{bertolinidarmon14}, noting that the Euler factors also match. Observe that in loc.\,cit. the authors restrict to the ordinary case, while e.g. in \cite[\S4.3]{BDV} the proof is adapted to general families.

Observe, however, that in this case we cannot discard the factors contributing to the adjoint, since they can be zero (and indeed, they vanish in the situations we want to study).
\end{proof}

Since the dependence on $\mathfrak f(k)$ is very minor and does not affect any of our results, we can (and we do) renormalize the adjoint $p$-adic $L$-function dividing by the fudge factor $\mathfrak f(k)$, and thus remove it from the statement of Prop. \ref{prop:factor-hr}.

\begin{remark}
We do not actually need the existence of the adjoint $p$-adic $L$-function in this setting. We can just claim an equality of the form  \[ L_p(\hf,\hg)(k,\ell,s) \times \mathfrak g(k)  = L_p(\hf,\chi_1)(k,s) \times L_p(\hf,\chi_2)(k,s-\ell+1), \] where $\mathfrak g(k)$ is a $p$-adic function interpolating the adjoint periods.
\end{remark}

As a piece of notation, we put $\hg = \hE$ to denote the Eisenstein family of characters $(\chi_1,\chi_2)$, provided that the characters are clear from the set-up. In general, in the case of even weights, we will not use the three variables $(k,\ell,s)$, and we often impose the restriction $s = \frac{k}{2} + \ell-1$. Hence, $L_p(\hf,\chi_2)(k,s-\ell+1) = L_p(\hf,\chi_2)(k,k/2)$ does not depend on $s$, and it is just a function on $k$. In this case,
\begin{equation}\label{eq:fact-2-var}
L_p(\hf,\hg)(k,\ell,s) \times L_p(\operatorname{Ad} \hf)(k) = \mathfrak f(k) \times L_p(\hf,\chi_1)(k,s) \times L_p(\hf,\chi_2)(k,k/2).
\end{equation}

We complement the previous result by discussing the Kato elements in the context of triple products and the Artin formalism, examining how the previous result can be interpreted as a factorization formula for a triple of modular forms $(f, g, h)$, where both $g$ and $h$ are Eisenstein. In particular, the choice $s = \frac{k}{2} + \ell - 1$ corresponds to taking $g$ and $h$ of the same weight. 

More precisely, let $f \in S_k$, $g = E_{\ell}(\chi_1,\chi_2)$ and $h = E_m(\xi_1,\xi_2)$. Write $c = \frac{k+\ell+m}{2}-1$ for the center of the $L$-function $L(f \otimes g \otimes h, s)$. Note that, by Artin formalism, $L(f \otimes g \otimes h,s)$ may be written as the product of the four $L$-functions attached to the twist of $f$ by the corresponding Dirichlet character: \[ L(f \otimes \chi_2 \xi_2, c) \cdot L(f \otimes \chi_1 \xi_2, c-\ell+1) \cdot L(f \otimes \chi_2 \xi_1, c-m+1) \cdot L(\chi_1 \xi_1,c-\ell-m+2). \] Proceeding as in Bertolini--Darmon \cite{bertolinidarmon14}, we may write the triple product $p$-adic $L$-function as the product of two Kitagawa--Mazur functions, using the appropriate functional equations \[ L_p(f, \chi_2 \xi_2, c) = L_p(f, \theta^{-1} \chi_2^{-1} \xi_2^{-1},k-c) = L_p(f, \chi_1 \xi_1, c-\ell-m+2) \] and that \[ L_p(f, \chi_1 \xi_2, c-\ell+1) = L_p(f, \theta^{-1} \chi_1^{-1} \xi_2^{-1},c-\ell+1) = L_p(f, \chi_2 \xi_1, c-m+1). \]

Observe indeed that this is consistent with the factorization formula of Proposition \ref{prop:factor-hr} when one takes $\ell = m$.

%Observe that the restriction $s = \frac{k}{2} + \ell-1$ corresponds to taking weights $\ell=m$.

%Consider the new variable $t = \frac{k-\ell+m}{2}$. Then, we are expressing the triple product $L$-function as the product of $L(f \otimes \chi_2 \xi_2,c)$ and $L(f \otimes \chi_1 \xi_2, t)$. Note that $s$ and $t$ must have different parity, which is equivalent to $k$ and $\ell$ having the same parity (or $m$ being even).

\subsection{Kato's Euler system}

For this section we follow the foundational work of \cite{kato04}, but using the language and conventions of more modern approaches like \cite{ochiai03}, \cite{bertolinidarmon14} or \cite{BB22}.

Recall that we have chosen bases $\gamma^{\pm}_{\hf}$ of the complex-conjugation eigenspaces $V(\hf)^{(c = \pm 1)}$, which are free of rank one over $A_U$. For the moment, we assume that we are outside the Eisenstein case.

\begin{definition}
Let \[ \kappa(\hf) \in H^1(\QQ, V(\hf)^*(-\hj)) \] denote the canonical \emph{Kato class} attached to the bases $\gamma_{\hf}^{\pm}$, which is the ``$p$-direction'' of an Euler system. 
\end{definition}

This class is constructed by interpolating the cup product of two modular units (or Eisenstein classes when the weight is greater than two), and then projecting to the corresponding isotopic component. The Kato class also depends on an auxiliary Dirichlet character, that we have denoted by $\chi$.

\begin{remark}
The construction of the Kato system depends on the choice of auxiliary integers $(c,d)$, relatively prime to $6Np$, where $N$ is the tame level of the family. However, it is possible to normalize the class in order to avoid that dependence. See e.g. the discussion of \cite[App. B]{BB22}. Observe that this is another point where the theory differs from the Beilinson--Flach case, where the factor cannot be always inverted (there are issues with that in the self-dual case), as it is discussed in \cite{KLZ17}.
\end{remark}

Kato's explicit reciprocity law \cite{kato04} establishes that the image of that class under the Perrin-Riou map recovers the $p$-adic $L$-function. That result requires the following definition, which may be seen as a measure of the contribution of the {\it bad} primes. We follow for this part \cite[\S6.6.3]{BB22}, where they discuss this issue in terms of an {\it imprimitive} $p$-adic $L$-function

\begin{definition}\label{def:bad-Euler}
We write $\mathcal E_N(x)$ for the two-variable $p$-adic $L$-function corresponding to the product of the {\it bad} primes: \[ \mathcal E_N(x) = \prod_{\ell \mid N} (1-a_{\ell}(\hf) \sigma_{\ell}^{-1}) \in A_U \hat{\otimes}_{\QQ_p} \cH_\Gamma. \]
\end{definition}

Since we are assuming that the prime $p$ does not divide $N$, this element must be understood as an analytic two-variable function in the weight and cyclotomic variables, which measures the difference between the motivic $p$-adic $L$-function and the analytic $p$-adic $L$-function (these factors are present in the Euler system, and may be inverted under some additional and mild assumptions).

More precisely, let $\cF^+ V(\hf)$ denote the rank 1 unramified subrepresentation of $V(\hf)$ over $A_U$ (which exists since $\hf$ is ordinary); and let $\eta_{\hf} \in \Dcris(\cF^+ V(\hf))$ be the canonical vector constructed in \cite{KLZ17}, which is characterized by interpolating the classes $\eta_f$ of \cref{sec:dRcoh} for each classical specialization $f$ of $\hf$. Then we have the following reciprocity law, where we have an extra contribution coming from the Euler factors at {\it bad} primes. For the following proposition, write $\mathcal{L}_{\cF^- V(\hf)^*}^{\PR}$ for the Perrin-Riou logarithm attached to $\cF^- V(\hf)^*$.

We state the reciprocity law in the framework of \cite{BB22}, where we avoid the critical Eisenstein points.

\begin{proposition}\label{prop:rec-law-normal}
Let $\hf$ be a Coleman family such that all its specializations are cuspidal (and which, in particular, does not pass through the critical $p$-stabilization of an Eisenstein series). Then, the following explicit reciprocity law holds:
\[ \left\langle \mathcal{L}_{\cF^- V(\hf)^*}^{\PR}( \loc_p \kappa(\hf)), \eta_{\hf} \right\rangle \cdot L_p(\operatorname{Ad} \hf) = \mathcal E_N(x) \cdot L_p(\hf,\chi) \cdot L_p(\hf). \]
\end{proposition}

\begin{proof}
This is Kato's explicit reciprocity law in the framework of Coleman families, as developed e.g. in \cite{BB22}. Observe that the presence of the adjoint $p$-adic $L$-function, and also the auxiliary Kitagawa--Mazur $p$-adic $L$-function $L_p(\hf,\chi)$ is usually treated as a unique contribution in terms of a non-vanishing factor, but we want to emphasize this dependence here since we will later need to be more circumspect around this point.
\end{proof}

\begin{remark}
There are some crucial differences between the theory of Beilinson--Flach and Beilinson--Kato classes. In particular, the latter depends on the choice of a suitable class in homology; the results of Ash and Stevens guarantee that this can be done in such a way that the Kato class is non-zero. Observe, however, that in the arguments of \cite{BB22} it is crucial that the critical point is cuspidal.
\end{remark}

Proposition \ref{prop:rec-law-normal} must be understood as an equality in two variables, with $L_p(\hf,\chi)$ playing an auxiliary role. In many case, it is possible to invert the factor $L_p(\hf,\chi)$ and normalize the Kato class dividing by it; however, in the Eisenstein case one must be more circumspect and we cannot do this normalization.

\begin{remark}
In the framework of triple products, when we had two ordinary Eisenstein series, one can formally the classes lie in \[ H^1(\QQ, V(f)^*(\chi_1^{-1} \xi_1^{-1})(1-(c-\ell-m+2))). \] Its image under the Perrin-Riou map is not just $L_p(f,\chi_1 \xi_1,c-\ell-m+2)$, as expected, since there is also the contribution coming from $L_p(f,\chi_1 \xi_2, c-\ell+m+1)$. If we further impose impose the condition $\chi_1 \xi_1 = 1$, and letting $t = (k+\ell-m)/2$, we have a class in $H^1(\QQ, V(f)^*(1-s))$ whose image is $L_p(f,\chi,t) \times L_p(f,s)$, suitably normalized by the adjoint.

When we specialize at the critical Eisenstein series, either $L_p(f, \chi, t)$ or $L_p(f, s)$ must vanish, but these two zeros are very different in nature, since the honest $p$-adic $L$-function for the representation is $L_p(f, s)$. In this part of the article, we consider the situation where $L_p(f, \chi, t)$ is zero.  
\end{remark}

In the setting where $\hf$ is not a critical Coleman family, the classes are usually renormalized to avoid dependence on the character $\chi$. More precisely, we may consider the normalized class
\[
\widetilde{\kappa(\hf)} = \frac{L_p(\operatorname{Ad} \hf)}{L_p(\hf, \chi)} \times \kappa(\hf).
\]

This class satisfies a reciprocity law that directly relates it to the $p$ -adic $L$ -function $L_p(\hf)$, thus avoiding the inconvenient appearance of the adjoint $p$-adic $L$-function and the auxiliary Kitagawa--Mazur $p$-adic $L$-function. However, as we later discuss, we must exercise more caution in the Eisenstein setting, since both $p$-adic $L$-functions vanish at the critical point.

\section{Deformation of Beilinson--Kato elements}\label{sec:def}

The aim of this section is to study the two--variable Beilinson--Kato elements in the case where they pass through a critical Eisenstein family. In particular, we want to relate the natural projection of that cohomology class with the circular units discussed in Section \ref{sec:circular}. Since the aim of this section is to establish Theorem \ref{theorem-main}, we assume that $r$ is even, although many of the results also hold in the odd case.

In particular, as we have already emphasized, the results of Benois and B\"uy\"ukboduk do not hold in our case, and we need to adapt both the interpolation of the Beilinson--Kato class and the explicit reciprocity law.

\subsection{Set-up for the degeneration process}

Consider the Beilinson--Kato class attached to two modular forms. As before, let $f = E_{r+2}(\psi,\tau)$ stand for the Eisenstein series of weight $r+2$ and characters $(\psi,\tau)$, and let $f_{\beta}$ be its critical slope $p$-stabilization. Under the non-criticality conditions we have fixed, there is a unique Coleman family $\hf$ passing through $f_{\beta}$, defined over some affinoid disc $U \ni r$. We may suppose that for all integers $k \in U \cap \ZZ_{\ge 0}$ with $k \ne r$, the specialization $f_k$ is a non-critical slope cusp form.

We also consider the modules \[ V(\hf)^*(-\hj) := V(\hf)^* \hat \otimes_{\QQ_p} \mathcal H_{\Gamma}(-\hj) \] and, in a similar way, the compactly supported version \[ V^c(\hf)^*(-\hj) := V^c(\hf)^* \hat \otimes_{\QQ_p} \mathcal H_{\Gamma}(-\hj). \]

\subsection{Families over punctured discs}

In this section, we discuss how we can make sense of a Beilinson--Kato cohomology class $\kappa(\hf)$ attached to the Coleman family $\hf$. The next result establishes that we have such a family of cohomology classes, maybe with a simple pole at the Eisenstein point.

\begin{theorem}\label{thm:bk-families}
There exists a cohomology class \[ \kappa(\hf) \in H^1(\QQ, \frac{1}{X} V^c(\hf)^*(-\hj)) \] with the following interpolation property: if $k \geq 0$ is an integer with $k \ne r$, then we have \[ \kappa(\hf)(k) = \mathcal{BK}^{[f_k]} \in H^1(\QQ, V(f_k)^* \otimes \mathcal H_{\Gamma}(-\hj)), \] where the element $\mathcal{BK}^{[f_k]} = \mathcal{BK}^{[f_k]}_{1, 1}$ is as defined e.g. in Theorem 9.2.1 of \cite{KLZ17}.
\end{theorem}

\begin{proof}
To begin with, observe that the results of \cite{BB22} do not apply in this setting, since, in particular, the key Prop. 6.5 of loc.\,cit. requires all the specializations being cuspidal, as it occurs in the Beilinson--Flach case in \cite{loefflerzerbes16}.

The result follows verbatim the reasoning of \cite[Thm. 5.4.2]{loefflerzerbes16}, which establishes an analogous result when all integer-weight specializations of $\hf$ are classical. In this case, the result does not apply for $k = r$, but inverting $X$ resolves the issue. In particular, the analysis of the comparison maps in cohomology developed in \cite[\S A]{LR1} makes clear that we are only introducing a simple pole at the critical Eisenstein point.
\end{proof}

The result asserts that we can directly interpolate the Beilinson--Kato classes in a neighborhood of the critical Eisenstein series, where the points are all cuspidal, and that the class can only present a simple pole at $f_{\beta}$. Next, we consider the projection of $\kappa(\hf)$ in cohomology.

\begin{definition}
We write $\kappa(f_{\beta})$ for the image of $\kappa(\hf)$ in the cohomology of the quotient \[ \frac{\frac{1}{X}V^c(\hf)^*}{V(\hf)^*} \cong \QQ_p(\tau^{-1})(1+r) \otimes \mathcal H_{\Gamma}(-\hj). \]
\end{definition}  

Note that this class is not zero a priori, as it happened in the scenarios discussed in \cite{LR1}, where the key point was that the class satisfied certain {\it good} local properties.

\subsection{The $p$-adic $L$-function and a factorization formula}

We being by recalling the following definition regarding the logarithmic distribution.

\begin{definition}
For a continuous character $\sigma : \ZZ_p^{\times} \rightarrow \CC_p$, the function $\frac{d^k \sigma}{dz^k} \cdot \frac{z^k}{\sigma(z)}$ is a constant, and $\log^{[k]} \in \CC_p$ is defined to be this constant. 
\end{definition}
An equivalent definition for $\log_p^{[k]}$  is $\log_p^{[k]}(\sigma) = w(\sigma)(w(\sigma)-1) \cdots (w(\sigma)-k+1)$, where $w(\sigma) = \frac{\log_p(\sigma(1+p))}{\log_p(1+p)}$. The analytic function $\log_p^{[k]}$ vanishes to order 1 at the characters $\sigma$ of the form $z \mapsto z^j \chi(z)$, where $\chi$ is a finite order Dirichlet character and $j$ is an integer such that $0 \leq j \leq k-1$.

\begin{remark}
Following the analogy with the Beilinson--Flach case, we would expect the class $\kappa(f_{\beta})$ to be divisible by the logarithmic distribution $\log^{[r+1]} \in \mathcal H_{\Gamma}$. Let us sketch why the same argument that we had used in that case does not work here. The specialization of $\kappa(\hf)$ at a locally-algebraic character of $\Gamma$ of degree $j \in \{0, \dots, r\}$ factors through the image of $\mathscr{D}_{\lambda_V - j}^s \otimes \operatorname{TSym}^j$ in $\mathscr{D}_{\lambda_V - (r+1)}^s \otimes \operatorname{TSym}^{(r+1)}$, and the maps $\Pr_{\hf}^{[j]}$ and $\Pr_{\hf}^{[r+1]}$ agree on this image up to a non-zero scalar. Since the $\Pr_{\hf}^{[j]}$ for $0 \le j \le r$ do not have poles at $X = 0$, it follows that the specializations of $\kappa(\hf)$ at a pair $(r, \chi)$, for $\ell \ge r$ and $\chi$ locally-algebraic of degree $\in \{0, \dots, r\}$, interpolate the projections of the classical Beilinson--Kato classes to the $(E_{r+2}^{\crit})$-eigenspaces in classical cohomology. However, in this setting, we cannot ensure whether or not these classes arise as suitable projections of classes in the cohomology of $X_1(N)$. 
%By Zariski-density, the class specializes to 0 everywhere in $\{r \} \times \{\chi\}$. Since this holds for all $\chi$ of degree up to $r$, and these are exactly the zeroes of $\log^{[r+1]}$, the result follows.  
\end{remark}

Consider now the two-variable $p$-adic $L$-function $L_p(\hf) \in A_V \hat{\otimes}_{\QQ_p}\cH_{\Gamma}$ as introduced in Section \ref{sec:KM}. For this construction, the interpolation property applies also at $k = r$ without any further modifications; and here the complex $L$-function factors, via Artin formalism, as
\[
L(E_{r+ 2}(\psi, \tau), 1 + j) = L(\psi, 1+j) \cdot L(\tau, j - r).
\]

Moreover, the restriction of $L_p(\hf)$ to the $k = r$ fiber is uniquely determined by its interpolation property at crystalline points (we do not need finite-order twists), and after Bella\"iche--Dasgupta \cite{bellaichedasgupta15} we also have an ``Artin formalism'' factorization along the region where $\psi(-1)$ agrees with the parity of $j$. 

\begin{definition}
Let $\lambda$ denote the function defined over the Iwasawa algebra such that its specialization at a character $\sigma$ corresponding to an integer $s$ and a Dirichlet character $\xi$ of $p$-power conductor is $\xi^{-1}(N_2) \langle N_2 \rangle^{-s+1}$.
\end{definition}

As usual, we write $f_{\beta} = E_{r+2}^{\crit}(\psi,\tau)$ for the critical $p$-stabilization of $E_{r+2}(\psi,\tau)$. Note that $L_p(f_{\beta},s)$ may be seen as the (one-variable) $p$-adic $L$-function corresponding to the specialization of the (two-variable) function $L_p(\hf)$.

\begin{theorem}\label{thm:bd}
 With the previous notations, we have the following:
 \begin{enumerate}
     \item[(a)] $L_p(f_{\beta},s) = 0$, if $(-1)^s = \psi(-1)$.
     \item[(b)] Up to an explicit rational factor, \[ L_p(f_{\beta},s) = \lambda(s) \log_p^{[r+1]}(s) L_p(\psi,s) L_p(\tau,s-r-1), \] if $(-1)^s = -\psi(-1)$.
 \end{enumerate}
 \end{theorem}

\begin{proof}
See \cite[\S5, 6]{bellaichedasgupta15}. More concretely, the proof of the first statement relies on the application $\Theta_k$, defined in \cite{bellaiche11b}, which allows us to relate the critical weight $k$ Eisenstein series with the ordinary Eisenstein series of weight $-2 - k$.
\end{proof}

Furthermore, the factor $L_p(\tau, s - k + 1)$ is the image under the Coleman map of the system of circular units, thus providing a natural link with the other Euler system explored in this note.

However, as we have already emphasized, the $p$-adic $L$-function that appears in the reciprocity law is not this one, but rather the Hida--Rankin $p$-adic $L$-function, which outside the critical line factors as the product of the two Kitagawa--Mazur $p$-adic $L$-functions $L_p(\hf)$ and $L_p(\hf, \chi, s)$, suitably normalized by the adjoint $p$-adic $L$-function.

The product of the two Kitagawa--Mazur $p$-adic $L$-functions always vanishes at the critical Eisenstein point, since one of them lies in the region of the weight space where Theorem \ref{thm:bd} guarantees that the $p$-adic $L$-function is zero. It is important to note that this does not imply that the image of the Beilinson--Kato class is zero, as this zero is canceled by the adjoint $p$-adic $L$-function. Hence, by combining Proposition \ref{prop:factor-hr} with Theorem \ref{thm:bd}, we obtain the following corollary. As a piece of notation, we write $L_p'(f_{\beta},\chi)$ for the derivative of the one-variable $p$-adic $L$-function $L_p(\hf,\chi)(k,k/2)$ at the specialization $f_{\beta}$; similarly, $L_p'(\operatorname{Ad} f_{\beta})$ denotes the derivative of $L_p'(\operatorname{Ad} \hf)$ at $f_{\beta}$. For the following result, recall that we are working under the assumption that $r$ is even.

\begin{corollary}\label{prop:factor-full}
Let $\hE$ be the (ordinary) Eisenstein family of characters $(1,\chi)$, and set $s = \frac{r}{2}+\ell$. Then, the following factorization formula holds when $(-1)^{r/2} (-1)^{\ell} = \psi(-1)$: \[ L_p(f_{\beta},\hE)(\ell,s) \times L_p'(\operatorname{Ad} f_{\beta}) = \mathfrak f(r) \cdot \log_p^{[r+1]}(s) \cdot L_p'(f_{\beta},\chi) \times L_p(\psi,s) \times L_p(\tau,s-r-1). \]
\end{corollary}

\begin{proof}
This follows by taking derivatives along the weight direction in Proposition \ref{prop:factor-hr}, and using that both $L_p'(\operatorname{Ad} \hf)$ and $L_p(\hf,\chi)$ vanish at the critical Eisenstein point. Then, specializing at $k=r$ and applying Theorem \ref{thm:bd}, the result follows.
\end{proof}

The previous result gives a precise description of $L_p(f_{\beta},\hE)$ as the product of a derivative of a Kitagawa--Mazur $p$-adic $L$-function and two Kubota--Leopoldt $p$-adic $L$-functions, normalized by the adjoint $p$-adic $L$-function. Compare e.g. with \cite[Thm. 3.3]{RiRo}, where a similar result was derived in the setting of congruences between cuspidal forms and Eisenstein series.

\subsection{Perrin-Riou maps}\label{sec:PR}

Once we have introduced the different $p$-adic $L$-functions and the Euler system, we aim to relate $L_p(\hf)$ to the image of the localization of the Beilinson--Kato class, $\operatorname{loc}_p(\kappa(\hf))$, under the projection to $\cF^- V(\hf)^*$.

At this point, and for establishing an explicit reciprocity law, the theory presents notable differences with that of Beilinson--Flach elements. In particular, the Beilinson--Kato element depends on the choice of a homology period. By a result of Ash--Stevens \cite[Thm. 13.6]{kato04} in the cuspidal case, there exists a choice of periods which guarantees that the class is non-zero. Further, the cohomology class obtained for a different choice of periods gives another class that, for any specialization, is a multiple of the one coming from Ash--Stevens. Unfortunately, in the Eisenstein case, we cannot claim that such a good normalization exists, so it is possible that all the possible choices give the zero class. Then, our results of this section must be interpreted as follows: if there exists a choice of period for which $\kappa(\hf)$ is not zero, we prove a reciprocity law for that class, which later allows us to relate it with a circular unit; if not, we just get the zero class.

\begin{remark}
This discussion suggests, as a very natural question, if we can prove that there exists a normalization of the Kato class such that $\kappa(\hf)$ is non-zero.
\end{remark}

Perrin-Riou's regulator gives us a map
\[
  \operatorname{Col}_{\mathbf{a}_{\hf}^-} = \left\langle \mathcal{L}^{\PR}_{\cF^- V(\hf)^*}(-), \mathbf{a}_{\hf}^- \right\rangle :
  H^1_{\mathrm{Iw}}(\QQ_{p, \infty}, \cF^- D(\hf)^*) \to \cH_{\Gamma}\hat{\otimes}_{\QQ_p}A_U
\]
which interpolates the Perrin-Riou regulators for $f_k$ for varying $k$. Let $\varphi^{-1}$ stand for the left inverse of the Frobenius, denoted as $\psi$ in \cite[\S8.2]{KLZ17}. More precisely, for $z \in \left(\cF^- D(\hf)^*\right)^{\varphi^{-1} = 1}$, this map sends $z$ to
\[ \langle\iota((1 - \varphi) z), \mathbf{a}_{\hf}^- \rangle, \]
 where $\iota$ is the inclusion
\[
\left(\cF^- D(\hf)^*\right)^{\varphi^{-1} = 0} \into \left(\cF^- D(\hf)^*[1/t]\right)^{\varphi^{-1} = 0}
= \Dcris\left(\cF^{-}D(\hf)^*\right) \otimes \cH_{\Gamma}.
\]

At the bad fiber, we have the relation
\[ \mathbf{a}^-_{\hf} \bmod X = d_r t^{r+1} \omega_{f_r} \otimes e_{-(r + 1)} \]
where $e_n$ is the standard basis of $\Zp(n)$ and $d_r$ is a constant. Since multiplication by $t^{r+1}$ corresponds to multiplication by $\log^{[r + 1]}$ on the $\cH_{\Gamma}$ side (see e.g. \cite{leiloefflerzerbes11}), we conclude that
\[ \operatorname{Col}_{\mathbf{a}_{\hf}^-}(\kappa(\hf)) \bmod X = d_r \cdot \log^{[r + 1]} \cdot  \left\langle \mathcal{L}^{\text{PR}}_{\tau^{-1}} (\kappa(f_{\beta})), \omega_{f_k} \right\rangle.
\]

We now establish the explicit reciprocity law for a family $\hf$ through a critical Eisenstein point. Note that this is not automatic from our previous result, which allows us to claim it in a neighborhood of the critical Eisenstein point, thus leading to the present of some annoying constants.

In order to state the next result, we recover the setting of Section \ref{sec:HR}. In particular, we are working with two-variables $(k,s)$, which may be understood in terms of the variables $(k,\ell)$ corresponding to the weight specializations of the Coleman family $\hf$ and the ordinary Eisenstein family $\hE$, where, with the notations of Section \ref{sec:HR}, $(\chi_1,\chi_2)=(1,\chi)$. In order to simplify the arguments in the next result, we assume that $L_p'(f_{\beta},\chi) \neq 0$, although a minor modification of the leading term argument can solve the issue.

\begin{theorem}\label{theorem:rec-law-critical}
Under Assumption \ref{ass-1} and Assumption \ref{ass-2}, we have
\[ \operatorname{Col}_{\mathbf{a}_{\hf}^-}(\kappa(\hf)) \cdot L_p(\operatorname{Ad} \hf) = d_{\hf}(\hk) \cdot \mathcal E_N \cdot L_p(\hf,\chi,\hk/2) \cdot L_p(\hf), \]
where $d_{\hf}(\hk)$ is either the zero function, or, if not, it is a meromorphic function on $U$ alone, regular and non-vanishing at all integer weights $k \ge -1$ except possibly at $r$ itself, and regular at $k = r$. Recall that $\mathcal E_N$ is the product of the bad Euler factors introduced in Definition \ref{def:bad-Euler}.
\end{theorem}

\begin{proof}
It follows from the reciprocity laws for Rankin--Selberg convolutions (of cuspidal forms) that the quotient $ \operatorname{Col}_{\mathbf{a}_{\hf}^-}(\kappa(\hf))/ (\mathcal E_N \cdot L_p(\hf,\hE))$ is a function of $\hk$ alone, and this ratio is either identically zero or, following the discussion of \cite[\S7]{BB22}, it does not vanish at any integer $k \ge -1$ where $f_k$ is classical; it is equal to the fudge-factor $d_k$ defined above.

Moreover, since under the current assumptions $L_p(\hf,\hE)$ is well defined and non-zero along $\{ r \} \times \cW$, we conclude that $d_{\hf}(\hk)$ does not have a pole at $\hk$ (although it might have a zero there).
\end{proof}

\begin{remark}
The non-vanishing of $L_p'(f_{\beta},\chi)$ must be understood as a very mild assumption, since the character $\chi$ can be chosen among all the possible Dirichlet characters satisfying a parity constraint. Further, if $L_p'(f_{\beta},\chi)=0$, the function $d_{\hf}(\hk)$ can have a (not necessarily simple) pole; in that case, our results can be adapted considering the leading term of $d_{\hf}(\hk)$ in its Laurent expansion.
\end{remark}

Note that the right hand side of the previous theorem may be understood as a Hida--Rankin $p$-adic $L$-function, which factors as the function $L_p(\hf)$ involved in our study times an extra factor that we may usually neglect for a fixed specialization of $f$. In particular, the previous theorem may be refined following Corollary \ref{prop:factor-full}, taking derivatives and then applying the factorization formula for $L_p(\hf)$.

\begin{remark}
Observe that this result is only valid in half of the weight space. If we assume for simplicity that both $r$ and $\psi$ are even, the reciprocity law encodes in the $p$-adic $L$-function the values of classes in the Iwasawa cohomology of $\QQ_p(\tau^{-1})(1-j)$ when $j$ is even. This is coherent with the fact that, in this case, Iwasawa cohomology has rank zero for odd values of $j$ (that is, the cohomology class and the $p$-adic $L$-function vanish in those cases).

%In this case, the class $\kappa(\hf)$ lifts (uniquely) to $H^1(\QQ, V(\hf)^*)$, and thus has a well defined image \[ \hat{\kappa}(f_\beta) \in H^1(\QQ, \Qp(\tau^{-1}) \otimes \cH_{\Gamma}(-\hj)).\] Iwasawa cohomology has now rank one for odd values of $j$. However, we do not have a reciprocity law in this case to relate it with other Euler systems.
\end{remark}

\subsection{Meromorphic Eichler--Shimura}

Before comparing the specialization of the Beilinson--Kato class with the system of circular units, we need to understand the variation of the canonical differential $\eta_{\hf}$ around the critical Eisenstein point. This is achieved in the following result.

\begin{proposition}
Under the running assumptions, there exists an integer $n \geq 0$ and a unique isomorphism of $A_U$-modules \[ \eta_{\hf} \colon \Dcris(\mathcal F^-D(\hf)^*) \cong X^{-n}  A_U \] whose specialization at every $k \geq 1 \in U$, with $k \neq r$, is the linear functional given by the pairing with the differential form $\eta_{f_k}$ associated to the weight $k$ specialization of $\hf$. For this $\eta_{\hf}$, we have \[ \left\langle \mathcal{L}^{\mathrm{PR}}_{\cF^-}(-), \eta_{\hf}^- \right\rangle \cdot L_p(\operatorname{Ad} \hf) = \mathcal E_N \cdot L_p(\hf,\chi,\hk/2) \cdot L_p(\hf). \]
\end{proposition}

\begin{proof}
We define $\eta_{\hf}$ to be the quotient $a_{\hf}^-/d_{\hf}$, and $n$ is the order of vanishing of $d_{\hf}$ at $k=r$.
\end{proof}

If $d_{\hf}(r) \ne 0$, then we have thus constructed a class in the Iwasawa cohomology of $V(\hf \times \tau)^*$ whose regulator agrees with the product of the circular unit Euler system for $\tau$, and a shifted copy of the Kubota--Leopoldt $p$-adic $L$-function for $\psi$.

We claim that if $d_{\hf}(r) = 0$, then in fact $\kappa(\hf)$ is divisible by $X$. If $d(r) = 0$, then $\kappa(f_{\beta})$ vanishes. So we can divide out a factor of $X$ from both $\kappa(\hf)$ and $d_{\hf}(\hk)$, and repeat the argument. Since $d_{\hf}$ is not identically 0, this must terminate after finitely many steps. Thus we have the following version of the {\it leading term argument}, following the ideas introduced in \cite{LZ-yoshida}.

\begin{proposition}\label{prop:leading-term}
Let $n \ge 0$ be the order of vanishing of $d_{\hf}$ at $k = r$ and suppose that Assumption \ref{ass-1} and \ref{ass-2} hold. Then $X^{-n} \kappa(\hf)$ is well defined and non-zero modulo $X$; and this leading term projects into the quotient $H^1(\QQ, \QQ_p(\tau^{-1}) \otimes \cH_{\Gamma}(-\hj))$. Its image under the Perrin-Riou regulator is given by a multiple of \[ \left( \frac{d^{*}_{\hf}(r) \cdot \mathcal E_N \cdot L_p'(f_{\beta},\chi)}{L_p'(\operatorname{Ad} f_{\beta})} \right) \cdot L_p(E_{r + 2}(\psi, \tau),\hj+1), \] where $d^{*}_{\hf}(r) \in L^\times$.
\end{proposition}

\begin{proof}
This follows directly from the previous discussion and the definition of $d^{*}_{\hf}(r)$.
\end{proof}

Observe that the quotient $\frac{L_p'(f_{\beta},\chi)}{L_p'(\operatorname{Ad} f_{\beta})}$ does not depend on the cyclotomic variable $\hj$, and the same is true for $d^{*}_{\hf}(r)$. Hence, this can be understood as a constant function over the Iwasawa algebra. Further, according to Theorem \ref{thm:bd}, \[ L_p(f_{\beta}, \hj+1) = \lambda \cdot \log^{[r+1]} \cdot L_p(\psi,\hj+1) \cdot L_p(\tau,\hj-r). \]

\begin{remark}
As we have already emphasized, Proposition \ref{prop:leading-term} can be adapted to the situation where $L_p'(f_{\beta},\chi)$ vanishes, by considering instead the product of the first non-vanishing term in the Taylor expansion and the leading term in the Laurent expansion of $d_{\hf}(\hk)$.
\end{remark}

\begin{definition}
Let \[ \kappa^*(f_{\beta}) \in H^1(\QQ, \QQ_p(\tau^{-1}) \otimes \cH_{\Gamma}(-\hj)) \] stand for the class obtained from $\kappa(\hf)$ after multiplying it by $X^{-n}$. 
\end{definition}

\subsection{Comparison with circular units}

In this section, we finally establish the main theorem of this note, relating the specialization of the Beilinson--Kato at the critical Eisenstein point with the system of circular units.

To establish the main result, we want to impose that $(-1)^r = \psi(-1)$. This assumption implies that $\tau$ is even and gives two possibilities for $\psi$ and $r$: either $r$ is even and $\psi$ is even, or $r$ is odd and $\psi$ is odd. If $r$ is odd and $\tau$ is even, the circular unit lies in the minus part of the Iwasawa cohomology, and hence the results we have discussed about circular units do not apply. Hence, we also assume that $r$ is even.

\begin{assumption}\label{ass:sign}
The integer $r$ is even and the Dirichlet character $\tau$ is also even.
\end{assumption}

The following result justifies why we have considered $L_p'(\operatorname{Ad} f_{\beta})$ and $L_p'(f_{\beta},\chi)$.

\begin{proposition}
If Assumption \ref{ass:sign} holds true, then $L_p(\operatorname{Ad} f_{\beta}) = 0$ and the (one-variable) $p$-adic $L$-function $L_p(\hf,\chi)(k,k/2)$ vanishes at $k=r$.
\end{proposition}

\begin{proof}
The vanishing of $L_p(\operatorname{Ad} f_{\beta})$ follows from the aforementioned results of \cite[\S9]{eigenbook}. Further, $L_p(\hf,\chi)(k,k/2) = 0$ as a consequence of Theorem \ref{thm:bd}.
\end{proof}

In particular, this lead us to consider the derivatives of the above $p$-adic $L$-functions. It is an interesting question, that we have not being able to solve, to determine if these derivatives may be expressed in terms of {\it interesting} arithmetic objects related with the modular form.

%For the statement, we implicitly assume that $L_p(\operatorname{Ad} \hf)$ has a simple zero at $k=r$; if not, we must understand it by multiplying both sides of the equality by $L_p'(\operatorname{Ad} \hf)$, thus having a zero in the right hand side.

%\begin{theorem}
%We have the following equality in $H^1(\QQ, \QQ_p(\tau^{-1})(1+r) \otimes \cH_{\Gamma}(-\hj))$: \[ \kappa^*(f_\beta) = \left(C \cdot \mathcal E_N \cdot \log^{[r+1]} \cdot \frac{L_p'(f_{\beta}, \chi)\cdot L_p(\psi,\hj+1)}{L_p'(\operatorname{Ad} f_{\beta})}\right) \cdot c(\tau)(r), \] for some non-zero constant $C$, where $c(\tau)(r)$ is the corresponding specialization of the system of circular units attached to $\QQ_p(\tau^{-1})(1+r)$.
%\end{theorem}

%\begin{proof}
%This follows from the previous proposition, together with the explicit reciprocity law for circular units, that both of the cohomology classes we are considering have the same image under the regulator.
%\end{proof}

We can now state and prove the main result of the note.
\begin{theorem}
Assume that $r$ is even and that both $\psi$ and $\tau$ are even. Under Assumption \ref{ass-1} and Assumption \ref{ass-2}, the following equality holds in $H^1(\QQ, \QQ_p(\tau^{-1})(1+r) \otimes \mathcal H_{\Gamma}(-\hj)))$: \[ \kappa^*(f_{\beta}) = \left( \frac{C \cdot \lambda \cdot \mathcal E_N \cdot L_p'(f_{\beta},\chi)}{L_p'(\operatorname{Ad} f_{\beta})} \right) \cdot L_p(\psi,\hj+1) c(\tau)(r), \] where $C$ is an explicit constant which does not depend on $j$.
\end{theorem}

\begin{proof}
We begin by noting that the Bloch--Kato conjecture implies that $H^1(\QQ, \QQ_p(\tau^{-1})(1+r) \otimes \mathcal H_{\Gamma}(-\hj)))$ is one-dimensional. Hence, it suffices to compare the image of both classes under the Perrin-Riou map presented in Section \ref{sec:PR}. 

Using next Theorem \ref{thm:bd}, the result is equivalent to checking the following equality between (two-variable) $p$-adic $L$-functions: \[ L_p(\hf,\hE) \cdot L_p'(\operatorname{Ad} f_{\beta}) = \mathcal E_N \cdot L_p'(f_{\beta},\chi) \cdot L_p(\hf). \] Here, $\hE$ is the Eisenstein family of characters $(1,\chi)$. The result follows by taking the derivative along the weight direction in the formula \[ L_p(\hf,\hE) \cdot L_p(\operatorname{Ad} \hf)(k) = \mathcal E_N \cdot L_p(\hf,\chi)(k,k/2) \cdot L_p(\hf), \] established in Section \ref{sec:HR}.
\end{proof}

\begin{remark}
Contrarily to which occurs in the Beilinson--Flach case, there is not a logarithmic contribution in the formula. In \cite{LR1} that contribution came from the class itself, since we had been able to argue, in geometric terms, that it was divisible by $\log^{[r+1]}$. Here, although a term of that kind arises in the factorization formula of the $p$-adic $L$-function, it is canceled out by the Coleman map.
\end{remark}

As we have already discussed, the terms $L_p'(f_{\beta},\chi)$ and $L_p'(\operatorname{Ad} f_{\beta})$ do not depend on the cyclotomic variable, so we may drop them from the statement. Roughly speaking, we can say that $\kappa^*(f_{\beta})$ agrees with $c(\tau)(r)$ multiplied by the Kubota--Leopoldt $p$-adic $L$-function of $\psi$.

\begin{corollary}
Assume that $r$ is even and that both $\psi$ and $\tau$ are even. Under Assumption \ref{ass-1} and Assumption \ref{ass-2}, the following equality holds in $H^1(\QQ, \QQ_p(\tau^{-1})(1+r) \otimes \mathcal H_{\Gamma}(-\hj))$: \[ \kappa^*(f_{\beta}) = \tilde C \cdot \lambda \cdot \mathcal E_N \cdot L_p(\psi,\hj+1) \cdot c(\tau)(r), \] where $\tilde C$ is an explicit constant which does not depend on $j$.
\end{corollary}

\section{Vanishing of the cohomology class and a factorization formula}\label{sec:other-cases}

We emphasize some limitations of our method and propose natural extensions of this project. These extensions are related to changes in the parity of either $\tau$ or $\psi$. We address two distinct issues: a factorization formula \'a la Bella\"iche--Dasgupta when $\tau$ is odd; and the situation arising from the vanishing of the class $\kappa^*(f_{\beta}) \in H^1(\QQ, \QQ_p(\tau^{-1})(1+r) \otimes \mathcal H_{\Gamma}(-\hj))$.

\subsection{A conjectural factorization formula}

For our later use, we consider the following assumption, that we will assume for most of the results of this section.

\begin{assumption}\label{ass:sign2}
The integer $r$ is even and the Dirichlet character $\tau$ is odd.
\end{assumption}

Two additional cases remain to be discussed, corresponding to the less common scenario where the integer $r$ is odd. We explain how our results can be adapted to this setting towards the end of the section.

Under Assumption \ref{ass:sign2}, it is natural to wonder if one may get a result analogous of Bella\"iche--Dasgupta to express the weight derivative $L_p'(f_{\beta})$ as the product of two Kubota--Leopoldt $p$-adic $L$-functions. This is motivated by the vanishing of $L_p(f_{\beta})$, which suggests the study of its derivative. This is explained in the following conjecture, whose proof seems beyond the techniques we have used in this work. At the end of the section we discuss some theoretical evidence towards it.

\begin{conjecture}\label{conj:fact}
If $\tau$ is odd--either under Assumption \ref{ass-2} or when $r$ is odd and $\tau$ is also odd--the following factorization formula holds: \[ L_p'(\hf)(r) = \lambda \cdot \log^{[r+1]} L_p(\psi,\hj+1) \cdot L_p(\tau,\hj-r), \] where $\lambda$ is the previously defined Iwasawa function and $L_p'(\hf)$ denotes the weight derivative of the $p$-adic $L$-function.
\end{conjecture}

%We believe that a proof of this statement must make use of the arithmetic of the companion ordinary Eisenstein series of weight $-2-k$.

It is interesting to observe the analogy with the treatment of \cite[\S5]{BB24}, where the higher order derivatives of the $p$-adic $L$-function $L_p(\hf)$ are studied. However, in loc.\,cit. the vanishing is due to the ramification of the eigencurve (in the cuspidal case, $e \geq 2$), and one can use the idea of Bella\"iche of defining {\it secondary $p$-adic $L$-functions}. Here, the vanishing of $L_p(\hf)$ at the Eisenstein point is due not to the ramification, but to the vanishing of the corresponding eigenspace of modular symbols. However, it still makes sense to try to capture the information of the Eisenstein representation as the sum of two characters.

\subsection{Vanishing of the Kato class}

Under Assumption \ref{ass:sign2}, we may proceed as in previous sections and take the projection of $\frac{1}{X}V^c(\hf)$ onto the quotient \[ \frac{\frac{1}{X}V^c(\hf)^*}{V(\hf)^*} \cong \QQ_p(\tau^{-1})(1+r) \otimes \mathcal H_{\Gamma}(-\hj). \] Considering the projection of the Kato class in cohomology, we obtain a class \[ \kappa(f_{\beta}) \in H^1(\QQ, \QQ_p(\tau^{-1})(1+r) \otimes \mathcal H_{\Gamma}(-\hj)).\] In this section, with a slight abuse of notation, we see the case just along half of the discs of the weight space, and still use the same notation: \[ \kappa(f_{\beta}) \in H^1(\QQ, \QQ_p(\tau^{-1})(1+r) \otimes \mathcal H_{\Gamma}(-\hj_+)).\] Then, the Kato class is also a well defined class, and proceeding as in the previous section, we may refine it to obtain \[ \kappa^*(f_\beta) \in H^1(\QQ, \Qp(\tau^{-1}) \otimes \cH_{\Gamma}(-\hj_+)).\]

Perrin-Riou's regulator, with $f$ now fixed as $E_{r+2}(\psi,\tau)$, gives us a (one-variable) map
\[
\operatorname{Col}_{b_f^+} = \left\langle \mathcal{L}^{\mathrm{PR}}_{\cF^+}(-), b_{f}^+ \right\rangle :
H^1_{\mathrm{Iw}}(\QQ_{p, \infty}, \cF^+ D(f)^*) \to \cH_{\Gamma}.
\]

In the reciprocity law, the roles of the auxiliary $p$-adic $L$-functions $L_p(\hf,\chi)$ and $L_p(\hf)$ are switched now: their product is still zero, but now $L_p(\hf,\chi,s)$ is a priori non-vanishing (or at least, it does not vanish due to a {\it trivial} sign argument), while $L_p(\hf)$ is zero using the aforementioned results of \cite{bellaichedasgupta15}. Hence, we have the following proposition.

\begin{proposition}
Under Assumption \ref{ass:sign2}, if $L_p(f_{\beta},\chi,r/2+1)$, we have the following equality:
\[ \operatorname{Col}_{b_f^+}(\kappa^*(f_\beta)) \cdot L_p'(\operatorname{Ad} f_{\beta})  = c(\hk) \cdot L_p(f_{\beta},\chi,r/2+1) \cdot L_p'(f_{\beta},\hj_+1), \] where $c(\hk)$ is a meromorphic function on $U$ alone, which is identically zero or which is regular and non-vanishing at all integer weights $k \ge -1$ except possibly at $r$ itself, and regular at $k = r$.
\end{proposition}

\begin{proof}
This follows by taking derivatives in the reciprocity law for Beilinson--Kato elements.
\end{proof}

When $\tau$ is odd and $\tau(p) \neq 1$, we can directly assert that the Beilinson--Kato classes vanish.

\begin{proposition}
When the character $\tau$ is odd and $\tau(p) \neq 1$, the space $H^1(\QQ, \QQ_p(\tau^{-1})(1+r) \otimes \mathcal H_{\Gamma}(-\hj_+))$ is zero and the Beilinson--Kato $\kappa^*(f_\beta) \in H^1(\QQ, \Qp(\tau^{-1}) \otimes \cH_{\Gamma}(-\hj_+))$ class vanishes. 
\end{proposition}

\begin{proof}
The Beilinson--Kato classes are universal norms, so it is a standard property that they are unramified outside $p$. If $\tau(p) \neq 1$, a standard computation shows that the dimension of the space is zero.
\end{proof}

This also means that either $L_p(f_{\beta},\chi,r/2+1)$ or $L_p'(f_{\beta},\hj_+1)$ must vanish. Further, since the projection of $\kappa(\hf)$ to the cohomology of the quotient \[ \frac{\frac{1}{X}V^c(\hf)^*}{V(\hf)^*} \cong \QQ_p(\tau^{-1})(1+r) \otimes \mathcal H_{\Gamma}(-\hj_+) \] vanishes, we may consider the lift \[ \kappa(\hf)_2 \in H^1(\QQ, V(\hf) \otimes \mathcal H_{\Gamma}(-\hj_+)).\]

\begin{definition}
We write $\kappa(f_{\beta})_2 \in H^1(\QQ,\QQ_p(\psi^{-1}) \otimes \mathcal H_{\Gamma}(-\hj))$ for the projection of $\kappa(\hf)_2$ onto the cohomology of the quotient \[ \frac{V(\hf)^*}{V^c(\hf)^*} \cong \QQ_p(\psi^{-1}) \otimes \mathcal H_{\Gamma}(-\hj_+). \]
\end{definition}

This definition does not only make sense in the case where $\tau$ is odd, but also in the even cases where $\hat \kappa(f_{\beta})$ vanishes (e.g. when $r$ is odd and $\tau$ is odd).

\begin{remark}
Observe that this object is analogous to the class $\kappa_{f,2}$ introduced in \cite{RiRo}, where it was described borrowing the tools introduced by Fukaya and Kato \cite{fukayakato12} in the study of Sharifi's conjectures \cite{sharifi12}.
\end{remark}

Recall that in previous section the Kato class was compared with the circular unit encoded in $L_p(\tau,\hj-r)$. When $r$ is odd and $\psi$ is even, it would be natural to expect a comparison with the circular unit $c(\psi)(-1)$, encoded in $L_p(\psi,\hj+1)$. In particular, Conjecture \ref{conj:fact} suggests an explicit comparison between this Kato class and that unit, that we expect to further develop in future work.

\subsection{The case of odd weight}

We still need to consider the cases where $r$ is odd. If $\tau$ is even and $\psi$ is odd, the same arguments explored in this section apply (provided that $\tau(p) \neq 1$), since we are in the part of Iwasawa cohomology that vanishes. Finally, if $\tau$ is odd and $\psi$ is even, the class \[ \kappa(f_{\beta}) \in H^1(\QQ, \QQ_p(\tau^{-1}) \otimes \cH_{\Gamma}(-\hj_+)) \] is a priori non-trivial (or, at least, it is not zero due to trivial sign reasons), but $L_p(\hf) = 0$, since the factorization formula of Bella\"iche--Dasgupta does not apply. In this case, we are unable to prove any satisfactory result.

\begin{remark}
We have emphasized that we are working with $\cH_{\Gamma}(-\hj_+)$, since twisting by odd character can indeed give non-zero values, for which one can apply the methods of Section \ref{sec:def} to relate them with circular units.
\end{remark}

Table \ref{table: parity} summarizes the four possible cases according to parity, assuming the absence of exceptional zeros. The BD column indicates with a check-mark whether the factorization formula of \cite{bellaichedasgupta15} (as a product of two Kubota--Leopoldt $p$-adic $L$-function) holds for $L_p(\hf)$. The columns $\kappa(f_{\beta})$ and $\kappa(f_{\beta})_2$ contain a check-mark if the corresponding elements $\kappa(f_{\beta})$ and $\kappa(f_{\beta})_2$ do not vanish for reasons related to parity.

\begin{table}[ht]
    \centering
    \renewcommand{\arraystretch}{1.5} % Para aumentar la altura de las filas
    \setlength{\tabcolsep}{10pt} % Ajusta la separación entre columnas
    \begin{tabular}{c|c|c|c|c|c}
        $r$ & $\psi$ & $\tau$ & BD & $\kappa(f_{\beta})$ & $\kappa(f_{\beta})_2$ \\
        \hline
        Even & Even & Even & \checkmark & \checkmark & \\
        Even & Odd & Odd & & & \checkmark \\
        Odd & Odd & Even & \checkmark & & \checkmark \\
        Odd & Even & Odd & & \checkmark & \\
    \end{tabular}
    \caption{Possible cases according to parity.}
    \label{table: parity}
\end{table}

The last case ($r$ odd, $\psi$ even and $\tau$ odd) sheds some light on the understanding of Conjecture \ref{conj:fact}. In this case, we can make sense of an Iwasawa cohomology class $c(\tau)$, obtained by $p$-adically interpolating the units corresponding to $\tau \xi$, where $\xi$ an odd character of $p$-power conductor. Hence, it makes sense to expect a comparison mimicking Theorem \ref{theorem-main}. In that case, $L_p(f_{\beta},\chi)$ is non-zero, while $L_p(f_{\beta})$ vanishes for the range of characters we are considering. This means that the derivative must somehow contain the information regarding $L_p(\tau,\hj-r)$ to carry out the comparison, thus supporting our conjecture. from a theoretical point of view. Hence, our results can be understood as follows:
\begin{itemize}
    \item When $\psi$ and $\tau$ are both even, we have established a comparison between $\kappa^*(f_{\beta})$ and $c(\tau)$.
    \item When $\psi$ is odd and $\tau$ is even, we have established that $\kappa(f_{\beta})$ vanishes, and we have defined a natural refinement of it, $\kappa(f_{\beta})_2$, which is conjecturally connected with $c(\psi)$.
    \item When $\psi$ is even and $\tau$ is odd, $\kappa(f_{\beta})$ is conjecturally linked to $c(\tau)$, but the relation depends on Conjecture \ref{conj:fact}.
    \item When $\psi$ is odd and $\tau$ is even, the two previous ideas must combine to give a formula for $\kappa(f_{\beta})_2$ in terms of $c(\psi)$.
\end{itemize}

A final comment is that the study of the odd case requires an alternative choice of twists: we cannot work with a triple of weights $(r+2,\ell,\ell)$ since the sum is not even, but this does not affect this discussion on the vanishing of the corresponding $p$-adic $L$-functions.

\begin{remark}
The trivial vanishing of $\kappa(f_{\beta})$ only depends on $r$ and $\tau$ (and not on $\psi$). However, the existence of a factorization formula as in \cite{bellaichedasgupta15} depends on $r$ and $\psi$ (and not on $\tau$).
\end{remark}

\begin{remark}
It remains to consider the case where $\tau$ is odd and $\tau(p)=1$. From the point of view of Selmer groups, the exceptional vanishing of the $p$-adic $L$-function $L_p(\tau,s)$  when $\tau(p)=1$ corresponds to the existence of an {\it extended Selmer group}, which in this case is just a $p$-unit. For simplicity, consider the case of $r=j=0$, and write $\kappa^*(f_{\beta})(1)$ for the specialization of the Iwasawa class $\kappa^*(f_{\beta})$ corresponding to $H^1(\QQ, \QQ_p(\tau^{-1})(1))$. As a piece of notation, let $H$ stand for the number field cut out by $\tau$ and let $\mathcal O_H^{\times}[1/p][\tau]$ denote the space of $p$-units in $H$ where the Galois action is via $\tau$. This space has a canonical generator, that we denote by $v_{\tau}$. This suggests the following question regarding how to interpret the possible torsion of Iwasawa cohomology.

\begin{question}
Can we have an explicit comparison between $\kappa^*(f_{\beta})(1)$ and the $p$-unit $v_{\tau} \in \mathcal O_H^{\times}[1/p][\tau]$.
\end{question}
\end{remark}

\section{Extensions of this work: Beilinson--Flach classes and the triple product case}

Along this note, we have shown how the comparison of Beilinson--Kato classes and circular units differs from the settings that appear in the previous work of Loeffler and the second author. We want to close the article with a short discussion of the differences between both methods, emphasizing the connections between both theories. For the main definitions and properties regarding Beilinson--Flach elements, we refer to \cite{KLZ17} and \cite{KLZ20}.

Let $\hf$ be a Coleman family passing through a critical Eisenstein series and let $\hg$ be a (ordinary) Hida family. When studying the comparison between Beilinson--Flach classes and Beilinson--Kato classes in \cite{LR1} we had used the Hida--Rankin $p$-adic $L$-function $L_p^g(\hf,\hg)$, and not the one corresponding to the situation where the critical Eisenstein family is dominant, that we would denote by $L_p^f(\hf,\hg)$. In the framework explored in this article, the only option was precisely considering that situation, since we just had one family. Hence, it is natural to extend these ideas to the Hida--Rankin setting and ask the following question (which also admits an analogue in the setting of triple products).

\begin{question}\label{q:fact}
Is there any factorization formula for $L_p^f(\hf,\hg)$ or its derivative along the weight-$k$ direction, similar to \cite{bellaichedasgupta15}?
\end{question}

Observe that the answer to this question is straightforward if we work instead with $L_p^g(\hf,\hg)$: applying the Artin formalism and specializing $\hf$ at $f = E_{r+2}(\psi,\tau)$, \[ L_p^g(E_{r+2},\hg)(\ell,s) = L_p(\hg,\psi)(\ell,s) \cdot L_p(\hg, \tau)(\ell,s-r-1). \] In the case of $L_p^f(E_{r+2},\hg)$, we expect to get some logarithmic contribution, as in \cite{bellaichedasgupta15}, together with a logaritmic term. The same question arises in the setting of triple products, where one would expect a similar situation for the triple product $p$-adic $L$-function of Andreatta--Iovita \cite{andreattaiovita21} when the dominant family passes through a critial Eisenstein point.

Moreover, in \cite{LR1} we had established that the projection of the Beilinson--Flach class to the quotient $\mathcal F^- V({\hf})^* \otimes V_{\hg}^*$ vanishes. This is because the Bloch--Kato conjecture holds for that representation and the projection of the Beilinson--Flach element satisfies a local condition at $p$ which is {\it too weak}.

Therefore, one may wonder if the specialization of the three-variable $p$-adic $L$-function $L_p^f(\hf,\hg)$ at $f_{\beta}$ is zero. A priori, it is not correct to say that the $p$-adic $L$-function is the image under the Perrin-Riou map of a trivial projection of the Beilinson--Flach class, since in the reciprocity law we may have a zero in the term $d_{\hf}(\hk)$ which arises in the leading term argument and one needs to be more cautious. Extending the techniques we have developed in this work, we are able to get a satisfactory answer of \ref{q:fact} in the second part of this paper, and also to understand better this issue.

%Alternatively, one may understand this result in terms of the interpolation property for the $p$-adic $L$-function $L_p^f(\hf,\hg)$. Hence, the more refined question we can propose is the following.

%\begin{question}
%Is there any factorization formula for $\frac{\partial L_p^f(\hf,\hg)}{\partial X}$, that is, the derivative along the weight-$k$-direction.
%\end{question}

\begin{remark}
This is similar to the kind of questions we had explored in \cite{rivero3}, where we also suggested that this kind of $p$-adic $L$-function may encode relevant information regarding the Galois representations. 
\end{remark}

Another interesting question, that we believe could be related with this discussion, is the study of the situation where two or more families pass through a critical Eisenstein series. These questions are also considered in the sequel to this work.

\begin{appendix}
    \section{Rank-2 big Galois representations}\label{section: big Galois representations}

\begin{center}
    by Ju-Feng Wu\footnote{J.-F.W. was supported by Taighde \'{E}ireann -- Research Ireland under Grant number IRCLA/2023/849 (HighCritical).}
\end{center}

The purpose of this appendix is to construct the big Galois representations $V(\mathbf{f})$ and $V^c(\mathbf{f})$ used in the main body of the paper. This is presumably well-known to the experts, but we have not been able to find a reference documenting such a construction in details.

Since the construction of these two big Galois representations requires notations in the construction of the cuspidal eigencurve, for the convenience of the reader, we will begin by reviewing the construction of $\mathcal{C}^0$ (Sect. \ref{subsection: oc cohomology} and \ref{subsection: cusp. eigencurve}). 
The construction of the two big Galois representations $V(\mathbf{f})$, $V^c(\mathbf{f})$ will then be presented afterwards (Sect. \ref{subsection: big Gal reps}). 
We should also remark that the constructions of $V(\mathbf{f})^*$ and $V^c(\mathbf{f})^*$ are similar, so we leave them to the reader.

\subsection{Overconvergent cohomology groups}\label{subsection: oc cohomology}
Recall the $p$-adic weight space $\mathcal{W}$, which is the rigid analytic space associated with the formal scheme $\mathrm{Spf} \, \Zp[\![\ZZ_p^\times]\!]$. Note that the algebra $\Zp[\![\ZZ_p^\times]\!]$ represents the functor (cf. \cite[Sect. 2.1]{CHJ-2017}) \[
    \mathsf{Alg}_{(\Zp, \Zp)} \rightarrow \mathsf{Sets}, \quad (R, R^+)\mapsto \Hom_{\mathsf{Groups}}^{\mathrm{cts}}(\ZZ_p^\times, R^\times),
\] 
from the category of complete sheafy $(\Zp, \Zp)$-algebras to the category of sets. We follow the ideas in \cite{CHJ-2017} and consider the following notions of \emph{weights}:

\begin{definition}\label{def: notions of weights}
    \begin{enumerate}
        \item A \emph{small $\Zp$-algebra} is a $\Zp$-algebra $R$ which is reduced, $p$-torsion-free, and finite as a $\Zp[\![X_1, ..., X_n]\!]$-algebra (for some $n\geq 0$). 
        \item A \emph{small weight} is a pair $(R_V, \lambda_V)$, where $R_V$ is a small $\Zp$-algebra and $\lambda_{V}: \ZZ_p^\times \rightarrow R_V^\times$ is a continuous group homomorphism such that $\lambda_V(1+p)-1$ is a topological nilpotent in $R_V$ with respect to the $p$-adic topology. 
        \item An \emph{affinoid weight} is a pair $(R_V, \lambda_{V})$, where $R_V$ is a reduced affinoid algebra (\`a la Tate) and $\lambda_V : \ZZ_p^\times \rightarrow R_V^\times$ is a continuous group homomorphism. 
        \item A \emph{weight} is either a small weight or an affinoid weight. 
        \item An \emph{open weight} is a weight such that the induced morphism of rigid analytic spaces \[
            V \coloneq \mathrm{Spa}(R_V, R_V^{\circ})^{\mathrm{rig}} \rightarrow \mathcal{W}
        \]
        is an open immersion. Here, $A_V^{\circ}$ is the ring of power-bounded elements in $R_V$ and `$\bullet^{\mathrm{rig}}$' stands for `taking the rigid generic fiber'.
    \end{enumerate}
\end{definition}

\begin{remark}\label{remark: some remarks on the notion of weights}
    \begin{enumerate}
        \item Let $(R_V, \lambda_V)$ be a weight. \begin{itemize}
            \item If $(R_V, \lambda_{V})$ is a small weight, then $R_V^{\circ} = R_V$. 
            \item If $(R_V, \lambda_V)$ is an affinoid weight, then $V = \mathrm{Spa}(R_V, R_V^{\circ})$.
        \end{itemize}
        \item Given a weight $(R_V, \lambda_V)$, $R_V[1/p]$ admits a structure of uniform $\QQ_p$-algebra by letting $R_V^{\circ}$ be its unit ball and equipping it with the corresponding spectral norm, denoted by $|\cdot|_{V}$. Following \cite[pp. 202]{CHJ-2017}, we then define \[
            s_V \coloneq \min\left\{ s\in \ZZ_{\geq 0} : |\lambda_{V}(1+p)|_{V} < p^{-\frac{1}{p^{s}(p-1)}}\right\}.
        \]
        \item Any $k\in \ZZ$ defines a small (resp., an affinoid) weight \[
            \ZZ_p^\times \xrightarrow{a\mapsto a^k} \ZZ_p^\times \quad (\text{resp., } \ZZ_p^\times \xrightarrow{a\mapsto a^k}\QQ_p^\times).
        \]
        In the later context, we shall switch between these two perspectives without mentioning. 
    \end{enumerate}
\end{remark}

Next, we need to introduce certain modules of analytic functions and distributions, which will be the coefficient system of the cohomology we are considering. We start by recalling some fundamental spaces in $p$-adic analysis: For any $s\in \QQ_{>0}$, denote by $C^{s}(\ZZ_p, \ZZ_p)$ the space of $s$-analytic functions on $\ZZ_p$ valued in $\ZZ_p$. By Amice's transform, we have the following isomorphism (cf. \cite[Chapter III, 1.3.8]{lazard65}) \[
    C^s(\ZZ_p, \ZZ_p) \cong \widehat{\bigoplus}_{i\in \ZZ_{\geq 0}} \ZZ_p e^{(s)}_i,
\]
where \[
    e^{(s)}_i: \ZZ_p \rightarrow \ZZ_p, \quad x \mapsto \lfloor p^{-s} i\rfloor! \begin{pmatrix}x\\i \end{pmatrix}.
\]

To define the modules of analytic functions and distributions, consider the following $p$-adic Lie group \[
    \mathrm{Iw}_{\GL_2} \coloneq \left\{ \gamma = \begin{pmatrix}a & b\\c & d\end{pmatrix} \in \GL_2(\ZZ_p): p \mid c\right\}.
\]

This is a Iwahori subgroup of $\GL_2(\ZZ_p)$ with Iwahori decomposition \[
    \mathrm{Iw}_{\GL_2} = N_1^{\mathrm{opp}} B(\ZZ_p),
\]
where $B$ is the upper-triangular Borel in $\GL_2$ and \[
    N_1^{\mathrm{opp}} = \left\{\begin{pmatrix} 1 & \\ c & 1\end{pmatrix} : c\in p\ZZ_p \right\}.
\]
Observe that \begin{equation}\label{eq: iso for p-adic manifolds}
    N_1^{\mathrm{opp}} \cong \ZZ_p, \quad \begin{pmatrix} 1 & \\ c & 1\end{pmatrix} \mapsto cp^{-1}
\end{equation}
as $p$-adic manifolds. 

For any weight $(R_{V}, \lambda_V)$ and $s\geq 1+s_V$, we define \begin{align*}
    A_{\lambda_V}^{s, \circ} & \coloneq \left\{ \phi: \mathrm{Iw}_{\GL_2}\xrightarrow{\text{continuous}} R_V: \begin{array}{l}
        \phi(\gamma\beta) = \lambda_V(\beta)\phi(\gamma) \text{ for all } (\gamma, \beta)\in \mathrm{Iw}_{\GL_2}\times B(\ZZ_p)  \\
         \phi|_{N_1^{\mathrm{opp}}} \in C^s(\ZZ_p, \ZZ_p)\widehat{\otimes}_{\ZZ_p}R_V^{\circ} \text{ via \eqref{eq: iso for p-adic manifolds}}
    \end{array} \right\}, \\
    A_{\lambda_V}^s & \coloneq A_{\lambda_V}^{s, \circ}[1/p], \\
    D_{\lambda_V}^{s, \circ} & \coloneq \Hom_{R_V^{\circ}}^{\mathrm{cts}}(A_{\lambda_V}^{s, \circ}, R_V^{\circ}), \\
    D_{\lambda_V}^{s} & \coloneq D_{\lambda_V}^{s, \circ}[1/p]. 
\end{align*}
Thanks to the structure theorem for $C^s(\ZZ_p, \ZZ_p)$, one easily sees that \begin{equation}\label{eq: analytic module structure}
    A_{\lambda_V}^{s, \circ} \cong \widehat{\bigoplus}_{i\in \ZZ_{\geq 0}} R_V^{\circ} e_i^{(s)} \quad  \text{ and } \quad D_{\lambda_V}^{s, \circ} \cong \prod_{i\in \ZZ_{\geq 0}} R_V^{\circ} e_i^{(s), \vee}
\end{equation}
as topological $R_V^{\circ}$-modules, where $e_i^{(s), \vee}$ is the dual vector to $e_i^{(s)}$. Moreover, the left multiplication of $\mathrm{Iw}_{\GL_2}$ induces a right $\mathrm{Iw}_{\GL_2}$-action on $A_{\lambda_V}^{s, \circ}$ and $A_{\lambda_V}^s$, which then further induces a left $\mathrm{Iw}_{\GL_2}$-action on $D_{\lambda_V}^{s, \circ}$ and $D_{\lambda_V}^s$.

If $(R_V, \lambda_V)$ is a small weight, we follow the idea in \cite{Hansen-Iwasawa} to give $D_{\lambda_V}^{s, \circ}$ a filtration. Let $I_V$ be an ideal of definition of $R_V$ and (again) assume $p\in I_V$. From the construction, the natural map $A_{\lambda_V}^{r-1, \circ} \rightarrow A_{\lambda_V}^{s, \circ}$ induces the dual map $D_{\lambda_V}^{s, \circ} \rightarrow D_{\lambda_V}^{s-1, \circ}$. We put \[
    \Fil^j \coloneq \ker \left( D_{\lambda_V}^{s, \circ} \rightarrow D_{\lambda_V}^{s-1, \circ}/I_V^j D_{\lambda_V}^{s-1, \circ} \right). 
\]
By carefully analyzing the structure theorem (cf. \cite{Hansen-Iwasawa}), one can show that each quotient \[
    D_{\lambda_V, j}^{s, \circ} \coloneq D_{\lambda_V}^{s, \circ}/\Fil^j
\]
is a finite abelian group and \[
    D_{\lambda_V}^{s, \circ} \cong \varprojlim_j D_{\lambda_V, j}^{s, \circ}
\]
as topological $R_V$-modules. Moreover, $\Fil^j$ is stable under the left $\mathrm{Iw}_{\GL_2}$-action and so $D_{\lambda_V, j}^{s, \circ}$ admits the induced left $\mathrm{Iw}_{\GL_2}$-action. 

Finally, we define the overconvergent cohomology groups in our consideration. Let $Y$ be the algebraic modular curve (over $\QQ$) of level $\Gamma = \Gamma_1(N) \cap \Gamma_0(p)$, where $p\nmid N$ is chosen as in the main body of the paper. We write $\overline{Y} = Y\times_{\mathrm{Spec}\, \QQ}\mathrm{Spec}\, \overline{\QQ}$. By fixing a geometric point $\overline{y}$ of $Y$, the finite abelian groups $D_{\lambda_V, j}^{s, \circ}$ (for some small weight $(R_V, \lambda_V)$ and some $s\geq 1+s_V$) define {\'e}tale local systems via the projection \[
    \pi^{\mathrm{{\acute{e}t}}}_1(\overline{Y}, \overline{y}) \twoheadrightarrow \mathrm{Iw}_{\GL_2}.\footnote{ More precisely, denote by $\overline{Y}_{\Gamma(p^n)}$ the modular curve of level $\Gamma_1(N)\cap \Gamma_0(p)$ over $\overline{\QQ}$, then the tower $\cdots \rightarrow \overline{Y}_{\Gamma(p^n)} \rightarrow \cdots \rightarrow \overline{Y}_{\Gamma(p)} \rightarrow \overline{Y}$ is a pro-{\'e}tale tower with Galois group $\mathrm{Iw}_{\GL_2}$.}
\] 
We denote by the resulting {\'e}tale local system by $\mathscr{D}_{\lambda_V, j}^{s, \circ}$. Then, we define \[
    H_{\mathrm{\acute{e}t}}^1(\overline{Y}, \mathscr{D}_{\lambda_V}^{s, \circ}) \coloneq \varprojlim_j H_{\mathrm{\acute{e}t}}^1(\overline{Y}, \mathscr{D}_{\lambda_V, j}^{s, \circ})\quad \text{ and }\quad H_{\mathrm{\acute{e}t}}^1(\overline{Y}, \mathscr{D}_{\lambda_V}^{s}) \coloneq H_{\mathrm{\acute{e}t}}^1(\overline{Y}, \mathscr{D}_{\lambda_V}^{s, \circ})[1/p].
\]

On the other hand, for any weight $(R_V, \lambda_V)$ and any $s\geq 1+s_V$, the left $\mathrm{Iw}_{\GL_2}$-action on $D_{\lambda_V, j}^{s, \circ}$, $D_{\lambda_V}^{s, \circ}$ and $D_{\kappa_V}^{s}$ define local systems on the complex manifold $Y(\CC)$ (\cite[Sect. 2]{ashstevens-arithcoh}), which are still denoted by $D_{\lambda_V, j}^{s, \circ}$, $D_{\lambda_V}^{s, \circ}$ and $D_{\lambda_V}^{s}$. Therefore, we can consider the Betti cohomology $H^1(Y(\CC), -)$ with these coefficients.

\begin{proposition}\label{prop: comparison of cohomology groups}
    Let $(R_V, \lambda_V)$ be a small weight and let $s\geq 1+s_V$. We have canonical isomorphisms \[
        H^1(Y(\CC), D_{\lambda_V}^s) \cong H_{\mathrm{\acute{e}t}}^1(\overline{Y}, \mathscr{D}_{\lambda_V}^{s}).
    \]
\end{proposition}
\begin{proof}
    Observe that \[
        H^1(Y(\CC), D_{\lambda_V}^s) = H^1(Y(\CC), D_{\lambda_V}^{s, \circ})[1/p].
    \]
    Hence, it is enough to show the existence of the canonical isomorphism \[
        H^1(Y(\CC), D_{\lambda_V}^{s, \circ}) \cong H^1_{\mathrm{\acute{e}t}}(\overline{Y}, \mathscr{D}_{\lambda_V}^{s, \circ}). 
    \]
    By Artin comparison (\cite[Theorem 2]{Artin-etale}), we have canonical isomorphisms \[
        H^1(Y(\CC), D_{\kappa_V, j}^{s, \circ}) \cong H^1_{\mathrm{\acute{e}t}}(\overline{Y}, \mathscr{D}_{\kappa_V, j}^{s, \circ}).
    \] 
    Hence, we only need to show the natural map \[
        H^1(Y(\CC), D_{\lambda_V}^{s, \circ}) \rightarrow \varprojlim_j H^1(Y(\CC), D_{\lambda_V, j}^{s, \circ})
    \]
    is an isomorphism. This follows from the fact that the inverse system $\{D_{\lambda_V, j}^{s, \circ}\}_j$ is Mittag-Leffler. 
\end{proof}

\begin{definition}\label{def: etale cohomology for affinoid weights}
    Suppose $(R_V, \lambda_V)$ is an affinoid weight such that there exists a small weight $(R_{V'}, \kappa_{V'})$ with an open immersion $V \hookrightarrow V'$. Let $s\geq 1+s_{V'}$. We define \[
        H^1_{\mathrm{\acute{e}t}}(\overline{Y}, \mathscr{D}_{\lambda_V}^s) \coloneq \left( \varprojlim_j H_{\mathrm{\acute{e}t}}^1(\overline{Y}, \mathscr{D}_{\lambda_{V'}, j}^{s, \circ}) \otimes_{\ZZ_p} R_{V}^{\circ}\right)[1/p].
    \]
\end{definition}

\begin{corollary}\label{cor: comparison of cohomology groups}
    Keep the notations as in Definition \ref{def: etale cohomology for affinoid weights}. Then, there exists a canonical isomorphism \[
        H^1(Y(\CC), D_{\lambda_V}^s) \cong H_{\mathrm{\acute{e}t}}^1(\overline{Y}, \mathscr{D}_{\lambda_V}^{s}).
    \]
    Consequently, Definition \ref{def: etale cohomology for affinoid weights} is independent to the choice of the small weight $(R_{V'}, \lambda_{V'})$. 
\end{corollary}
\begin{proof}
    By Proposition \ref{prop: comparison of cohomology groups}, we have \[
        \varprojlim_j \left(H^1(Y(\CC), D_{\lambda_{V'}, j}^s)\otimes_{\ZZ_p}R_V^{\circ}\right) \cong \varprojlim_j \left( H_{\mathrm{\acute{e}t}}^1(\overline{Y}, \mathscr{D}_{\lambda_{V'}, j}^{s, \circ}) \otimes_{\ZZ_p} R_{V}^{\circ} \right).
    \]
    Since $R_{V}^{\circ}$ is flat over $\ZZ_p$, the higher Tor groups $\mathrm{Tor}_i^{\ZZ_p}(-, R_V^{\circ})$ vanishes (\cite[\href{https://stacks.math.columbia.edu/tag/00M5}{Tag 00M5}]{stacks-project}) and so \[
        H^1(Y(\CC), D_{\lambda_{V'}, j}^s)\otimes_{\ZZ_p}R_V^{\circ} = H^1(Y(\CC), D_{\lambda_{V'}, j}^s \otimes_{\ZZ_p}R_V^{\circ}).
    \] 
    Observe also that $\{D_{\lambda_{V'}, j}^s \otimes_{\ZZ_p}R_V^{\circ}\}_{j}$ is a Mittag-Leffler projective system. Therefore, we arrive at the isomorphism \[
        H^1(Y(\CC), \varprojlim_j D_{\lambda_{V'}, j}^s \otimes_{\ZZ_p}R_V^{\circ}) \cong \varprojlim_j \left( H_{\mathrm{\acute{e}t}}^1(\overline{Y}, \mathscr{D}_{\lambda_{V'}, j}^{s, \circ}) \otimes_{\ZZ_p} R_{V}^{\circ} \right).
    \]
    To conclude the proof, it remains to show that \[
        D_{\lambda_V}^{s, \circ} \cong \varprojlim_{j} D_{\lambda_{V'}, j}^{s, \circ}\otimes_{\ZZ_p} R_{V}^{\circ}.
    \]
    Indeed, we have \[
        D_{\lambda_V}^{s, \circ} \cong \prod_{i\in \ZZ_{\geq 0}} R_V^{\circ} e_i^{(s), \vee} \cong \varprojlim_{j} D_{\lambda_{V'}, j}^{s, \circ}\otimes_{\ZZ_p} R_V^{\circ},
    \]
    where the first isomorphism follows from \eqref{eq: analytic module structure} and the last isomorphism follows from \cite[Proposition 6.4]{CHJ-2017}.
\end{proof}

\begin{remark}\label{remark: on comparison theorems}
    \begin{enumerate}
        \item Proposition \ref{prop: comparison of cohomology groups} and Corollary \ref{cor: comparison of cohomology groups} then allow us to equip a Galois action on $H^1(Y(\CC), D_{\lambda_V}^{s})$. In what follows, we shall often use these comparisons to jump between Betti cohomology and {\'e}tale cohomology. 
        \item The discussions above also apply to compactly supported cohomology. 
    \end{enumerate}
\end{remark}

\subsection{The cuspidal eigencurve and {\'e}tale points over the weight map}\label{subsection: cusp. eigencurve}
To construct the cuspidal eigencurve, let $(R_V, \lambda_V)$ be an affinoid weight, $s\geq 1+s_V$, and consider the natural morphism \begin{equation}\label{eq: Hc1-->H1}
    H_c^1(Y(\CC), D_{\lambda_V}^s) \rightarrow H^1(Y(\CC), D_{\lambda_V}^s).
\end{equation}
Following \cite{eigenbook}, define \begin{align*}
    H_{\mathrm{par}}^1(Y(\CC), D_{\lambda_V}^s) & \coloneqq \mathrm{image}\left( H_c^1(Y(\CC), D_{\lambda_V}^s) \rightarrow H^1(Y(\CC), D_{\lambda_V}^s) \right),\\
    H_{\mathrm{Eis}}^1(Y(\CC), D_{\kappa_V}^s) & \coloneqq \ker\left( H_c^1(Y(\CC), D_{\lambda_V}^s) \rightarrow H^1(Y(\CC), D_{\lambda_V}^s) \right). 
\end{align*}

Suppose $h\in \QQ_{\geq 0}$ such that $(U(p), V, h)$ is a slope datum (in the sense of \cite[Sect. 3.1]{hansen14}). Then, by \cite[Proposition 3.1.5]{hansen14}, $H_c^1(Y(\CC), D_{\lambda_V}^s)$ and $H^1(Y(\CC), D_{\lambda_V}^s)$ have slope-$\leq h$ decomposition and are independent to $s$. Since \eqref{eq: Hc1-->H1} is Hecke equivariant, one deduces that both  $H_{\mathrm{par}}^1(Y(\CC), D_{\lambda_V}^s)$ and $H_{\mathrm{Eis}}^1(Y(\CC), D_{\lambda_V}^s)$ also have slope-$\leq h$ decomposition. For $?\in \{\emptyset, c, \mathrm{par}, \mathrm{Eis}\}$, we define \[
    \mathbb{T}_{?, V}^{\leq h} \coloneq \text{ $R_V$-algebra in $\mathrm{End}_{R_V}(H_?^1(Y(\CC), D_{\lambda_V}^s)^{\leq h})$ generated by the Hecke algebra}.
\]
Since $H_?^1(Y(\CC), D_{\lambda_V}^s)^{\leq h}$ is a finite $R_V$-module, $\mathbb{T}_{?, V}^{\leq h}$ is a finite algebra over $R_V$, hence an affinoid algebra (\`a la Tate).

The \emph{cuspidal eigencurve} $\mathcal{C}^0$ is then constructing by gluing $\{\mathrm{Spa}(\mathbb{T}_{\mathrm{par}, V}^{\leq h}, \mathbb{T}_{\mathrm{par}, V}^{\leq h, \circ})\}_{(V, h)}$ over $\mathcal{W}$ (see \cite[Sect. 4.3]{hansen14} for details), running through the pair $(V, h)$ such that $V$ is an open affinoid weight and $(U(p), V, h)$ is a slope datum. The natural map \[
    \mathrm{wt}: \mathcal{C}^0 \rightarrow \mathcal{W}
\]
is the so-called \emph{weight map}.

In what follows, we study the modules $H_?^1(Y(\CC), D_{\lambda_V}^s)^{\leq h}$ locally around an {\'e}tale point $x\in \mathcal{C}^0(K)$ (for some finite extension $K$ of $\QQ_p$), which is the main interest in the main body of the paper. More precisely, suppose $x\in \mathcal{C}^0(K)$ such that $\mathrm{wt}$ is {\'e}tale at $x$. We write $\kappa = \mathrm{wt}(x)$. Let $V \subset \mathcal{W}$ be an open affinoid containing $\kappa$ and let $U$ be the connected component of $\mathrm{wt}^{-1}(U)$ containing $x$; and so $U$ corresponds to an idempotent $e_U: \mathscr{O}_{\mathcal{C}^0}(\mathrm{wt}^{-1}(V)) \twoheadrightarrow \mathscr{O}_{\mathcal{C}^0}(U)$. Since $\mathrm{wt}$ is {\'e}tale at $x$, after possibly shrinking $V$, we may assume \[
    \mathrm{wt}: U \rightarrow V
\]
is an isomorphism. Therefore, from the construction, we may then write \[
    \mathscr{O}_{\mathcal{C}^0}(U) = e_U \mathbb{T}_{\mathrm{par}, V}^{\leq h}
\]
for some $h\in \QQ_{\geq 0}$.

We construct the $R_V$-modules $e_U H^1(Y(\CC), D_{\lambda_V}^s)$ and $e_U H^1_c(Y(\CC), D_{\lambda_V}^s)$ as follows:

\paragraph{\textbf{Construction of $e_U H^1(Y(\CC), D_{\lambda_V}^s)^{\leq h}$.}} By definition, the map \eqref{eq: Hc1-->H1} yields a Hecke-equivariant inclusion \[
    H_{\mathrm{par}}^1(Y(\CC), D_{\lambda_V}^s)^{\leq h} \hookrightarrow H^1(Y(\CC), D_{\lambda_V}^s)^{\leq h}.
\]
Therefore, the restriction map gives the surjections \[
    \mathbb{T}_{V}^{\leq h} \twoheadrightarrow \mathbb{T}_{\mathrm{par}, V}^{\leq h} \twoheadrightarrow e_U \mathbb{T}_{\mathrm{par}, V}^{\leq h}.
\] We then define \[
    e_U H^1(Y(\CC), D_{\lambda_V}^s)^{\leq h} \coloneq H^1(Y(\CC), D_{\lambda_V}^s)^{\leq h} \otimes_{\mathbb{T}_{V}^{\leq h}} e_U \mathbb{T}_{\mathrm{par}, V}^{\leq h}.
\]

\paragraph{\textbf{Construction of $e_U H^1_c(Y(\CC), D_{\lambda_V}^s)^{\leq h}$.}} By definition, the map \eqref{eq: Hc1-->H1} yields a Hecke-equivariant surjection \[
    H^1_c(Y(\CC), D_{\lambda_V}^s)^{\leq h} \twoheadrightarrow H^1_{\mathrm{par}}(Y(\CC), D_{\lambda_V}^s)^{\leq h}. 
\]
Therefore, it gives rise to surjections \[
    \mathbb{T}_{c, V}^{\leq h} \twoheadrightarrow \mathbb{T}_{\mathrm{par}, V}^{\leq h} \twoheadrightarrow e_U \mathbb{T}_{\mathrm{par}, V}^{\leq h}.
\]
We then define \[
    e_U H^1_c(Y(\CC), D_{\lambda_V}^s)^{\leq h} \coloneq H^1(Y(\CC), D_{\lambda_V}^s)^{\leq h} \otimes_{\mathbb{T}_{c,V}^{\leq h}} e_U \mathbb{T}_{\mathrm{par}, V}^{\leq h}.
\]

\begin{proposition}\label{prop: incusion from eUH^1_c--> eUH^1}
    Keep the notations as above. The natural map \[
        e_U H^1_c(Y(\CC), D_{\lambda_V}^s)^{\leq h} \rightarrow e_U H^1(Y(\CC), D_{\lambda_V}^s)^{\leq h}
    \]
    is an injection. 
\end{proposition}
\begin{proof}
    By definition, the map \eqref{eq: Hc1-->H1} yields a Hecke-equivariant yields a Hecke-equivariant injection \[
        H_{\mathrm{Eis}}^1(Y(\CC), D_{\lambda_V}^s) \hookrightarrow H_c^1(Y(\CC), D_{\lambda_V}^s).
    \]
    Therefore, the restriction map gives rise to surjections \[
        \mathbb{T}_{c, V}^{\leq h} \twoheadrightarrow \mathbb{T}_{\mathrm{Eis}, V}^{\leq h}.
    \]
    That is, we may view $H_{\mathrm{Eis}}^1(Y(\CC), D_{\lambda_V}^s)$ as a $\mathbb{T}_{c, V}^{\leq h}$-module. We then define \[
        e_U H_{\mathrm{Eis}}^1(Y(\CC), D_{\lambda_V}^s)^{\leq h} \coloneq H_{\mathrm{Eis}}^1(Y(\CC), D_{\lambda_V}^s)^{\leq h} \otimes_{\mathbb{T}_{c, V}^{\leq h}} e_U\mathbb{T}_{\mathrm{par}, V}^{\leq h}.
    \]
    Consequently, we have a exact sequence \[
         e_U H_{\mathrm{Eis}}^1(Y(\CC), D_{\lambda_V}^s)^{\leq h} \rightarrow e_U H_{c}^1(Y(\CC), D_{\lambda_V}^s)^{\leq h} \rightarrow e_U H^1(Y(\CC), D_{\lambda_V}^s)^{\leq h}.
    \]
    However, since \[
        \frac{H_{c}^1(Y(\CC), D_{\lambda_V}^s)^{\leq h}}{H_{\mathrm{Eis}}^1(Y(\CC), D_{\lambda_V}^s)^{\leq h}} \cong H^1_{\mathrm{par}}(Y(\CC), D_{\lambda_V}^s)^{\leq h},
    \]
    one concludes that $e_U H_{\mathrm{Eis}}^1(Y(\CC), D_{\lambda_V}^s)^{\leq h} =0$. The desired injection then follows. 
\end{proof}

\begin{remark}
    Combing Proposition \ref{prop: comparison of cohomology groups} and Proposition \ref{prop: incusion from eUH^1_c--> eUH^1}, one sees that there exists a Hecke- and Galois-equivariant inclusion \[
        e_U H_{\mathrm{\acute{e}t}, c}^1(\overline{Y}, \mathscr{D}_{\lambda_V}^s)^{\leq h} \hookrightarrow e_U H_{\mathrm{\acute{e}t}}^1(\overline{Y}, \mathscr{D}_{\lambda_V}^s)^{\leq h}.
    \]
    Here, the slope-$\leq h$ decompositions are induced from  Proposition \ref{prop: comparison of cohomology groups}.
\end{remark}

\subsection{Rank-2 big Galois representations}\label{subsection: big Gal reps}
Let $f$ be a finite-slope eigenform of level $\Gamma_1(N)\cap \Gamma_0(p)$ and weight $k+2$. We make the following assumption: 

\begin{assumption}\label{assump: multiplicity-one}
    For $?\in \{\emptyset, c\}$, the Hecke-eigensystem defined by $f$ appears in $H^1_?(Y(\CC), D_k^s)$ and \[
        \dim H^i_?(Y(\CC), D_k^s)_{(\mathbb{T}=f)} = \dim H^i_?(Y(\CC), D_k^s)[\mathbb{T}=f] = \left\{ \begin{array}{ll}
            2, & \text{if }i=1 \\
            0, & \text{else}
        \end{array}\right. .
    \]
    Here, given a $\QQ_p$-vector space $M$ with a Hecke action, we write $M[\mathbb{T}=f]$ for the $f$-isotypic part and $M_{(\mathbb{T}=f)}$ for the generalized $f$-eigenspace.  
\end{assumption}

\begin{remark}\label{remark: multiplicity-one assumption}
    Since we are interested in critical-slope {\'e}tale points in the main body of the paper, by \cite[Theorem 1 \& Theorem 4]{bellaiche11a}, Assumption \ref{assump: multiplicity-one} is a reasonable assumption. Note that, as remarked in \cite[Sect. A4.2]{LR1}, \cite[Theorem 3.30]{bellaiche11a} also works for non-compactly supported cohomology. 
\end{remark}

\begin{proposition}\label{prop: big Galois representations}
    Keep the notations as above. Assume $f$ defines a point $x\in \mathcal{C}^0(K)$ (for some finite extension $K$ over $\QQ_p$) such that $\mathrm{wt}$ is {\'e}tale at $x$. Then, there exists an affinoid neighbourhood $U = \mathrm{Spa}(R_U, R_U^{\circ})\subset \mathcal{C}^0$ containing $x$ with corresponding Coleman family $\mathbf{f}$ and two Galois representations \[
        V(\mathbf{f}) \quad \text{ and }\quad V^c(\mathbf{f})
    \] over $R_U$ such that \begin{enumerate}
        \item the weight map gives an isomorphism $U \cong \mathrm{wt}(U)$;
        \item $V(\mathbf{f})$ and $V^c(\mathbf{f})$ are both free of rank $2$ over $R_U$;
        \item for any small-slope cuspidal eigenform $g$ that defines a point in $U$ and is not a $p$-stabilization of an Eisenstein series, the specializations of $V(\mathbf{f})$ and $V^c(\mathbf{f})$ at $g$ are the Galois representations \[
            V(g) = H^1_{\mathrm{\acute{e}t}}(\overline{Y}, \mathrm{Sym}^{\mathrm{wt}(y)}\QQ_p^2)[\mathbb{T}=g] \quad \text{ and }\quad V^c(g) = H^1_{\mathrm{\acute{e}t},c}(\overline{Y}, \mathrm{Sym}^{\mathrm{wt}(y)}\QQ_p^2)[\mathbb{T}=g]
        \]
        respectively; in fact, $V(g) = V^c(g)$.
    \end{enumerate}
\end{proposition}
\begin{proof}
    Since $\mathrm{wt}$ is {\'e}tale at $x$, the first assertion follows immediately. In what follows, we denote $\mathrm{wt}(U)$ by $V = \mathrm{Spa}(R_V, R_V^{\circ})$ and let $\lambda_V$ be the corresponding open affinoid weight. 

    Since $f$ defines a point $x\in \mathcal{C}^0(K)$, there exist $s \geq 1+s_V$ and $h\in \QQ_{\geq 0}$ such that $\mathscr{O}_{\mathcal{C}^0}(U) = e_U \mathbb{T}_{\mathrm{par}, V}^{\leq h}$. For $?\in \{\emptyset, c\}$, we define \[
        V^?(\mathbf{f}) = H_{\mathrm{\acute{e}t},?}^1(\overline{Y}, \mathscr{D}_{\lambda_V}^s)^{\leq h}[\mathbb{T}=\mathbf{f}] \coloneq e_U H_{\mathrm{\acute{e}t},?}^1(\overline{Y}, \mathscr{D}_{\lambda_V}^s)^{\leq h}
    \]
    To show that both $R_V$-modules are free over rank $2$, denote by $R_{V, k}$ the completed local ring at $k$ and let $\mathfrak{m}_k$ be the maximal ideal. Then, $H_{\mathrm{\acute{e}t},?}^1(\overline{Y}, \mathscr{D}_{\lambda_V}^s)^{\leq h}_{(\mathbb{T}=f)}$ is a finitely generated $R_{V, k}$-module. Note that \[
        H^i_?(Y(\CC), D_{\lambda_V}^s\otimes_{R_V} R_V/\mathfrak{m}_k)^{\leq h}_{(\mathbb{T}=f)} = H^i_?(Y(\CC), D_k^s)^{\leq h}_{(\mathbb{T}=f)}
    \]
    is concentrated in degree $1$ and has dimension $2$ when $i=1$ by Assumption \ref{assump: multiplicity-one}. Therefore, by \cite[Lemma 2.9]{barreradimitrovjorza}, we see that $H^1_?(Y(\CC), D_{\lambda_V}^s)_{(\mathbb{T}=f)}^{\leq h}$ is free of rank $2$ over $R_{V, k}$. Hence, after shrinking $U$ if necessary and applying Proposition \ref{prop: comparison of cohomology groups}, we conclude that $V^?(\mathbf{f})$ is a rank-$2$ Galois representation over $R_V$.

    Finally, the third assertion follows from specialization and the control theorem \cite[Theorem 3.2.5]{hansen14}.
\end{proof}

Combining Proposition \ref{prop: incusion from eUH^1_c--> eUH^1} and Proposition \ref{prop: big Galois representations}, we arrive at the following corollary: 

\begin{corollary}\label{cor: injection of big Galois representations}
    Keep the notations as above. There is a natural inclusion of Galois representations \[
        V^c(\mathbf{f}) \hookrightarrow V(\mathbf{f})
    \]
    over $R_{V}$.
\end{corollary}

\begin{remark}\label{remark: dual construction}
    We close this appendix by briefly discuss the constructions for $V(\mathbf{f})^*$ and $V^c(\mathbf{f})^*$. For any weight $(R_V, \lambda_V)$ and $s\geq 1+s_V$, let $w_p = \begin{pmatrix} 0 & -1\\ p & 0\end{pmatrix}$ and \[
        A_{\lambda_V}^{s, \circ, *} \coloneq \left\{\phi: w_p \mathrm{Iw}_{\GL_2} \rightarrow R_V: \begin{array}{l}
             \phi(\gamma\beta) = \lambda_V(\beta)\phi(\gamma) \,\, \text{for all }(\gamma, \beta)\in w_p\mathrm{Iw}_{\GL_2}\times B(\ZZ_p)   \\
             \phi|_{w_pN_1^{\mathrm{opp}}} \in C^r(\ZZ_p, \ZZ_p)\widehat{\otimes}_{\ZZ_p} R_V^{\circ}
        \end{array}\right\}.
    \]
    Then, it admits a twisted $\mathrm{Iw}_{\GL_2}$-action by \[
        (\gamma * f)(\alpha) = f(w_p \gamma w_p^{-1}\alpha).
    \]
    One can similarly define $A_{\lambda_V}^{s, *}$, $D_{\lambda_V}^{s, \circ, *}$, and $D_{\lambda_V}^{s, *}$ as before with the corresponding $\mathrm{Iw}_{\GL_2}$-action. By replacing $D_{\lambda_V}^s$ by $D_{\lambda_V}^{s, *}$, working with {\'e}tale cohomology groups $H^1_{\mathrm{\acute{e}t}, ?}(\overline{Y}, \mathscr{D}_{\lambda_V}^{s, *}(1))$ (for $?\in \{\emptyset, c\}$), and applying the strategy above with respect to the dual Hecke operator $U'(p)$, one then constructs $V(\mathbf{f})^*$ and $V^c(\mathbf{f})^*$. We leave the details to the reader. 
\end{remark}

\end{appendix}

 \bibliographystyle{amsalphaurl}
 \bibliography{references}

 \end{document}